\documentclass[11pt]{amsart}
\usepackage{pxfonts,multicol}
\usepackage[colorlinks=true,linkcolor=blue,urlcolor=blue,citecolor=blue, pagebackref]{hyperref}
\usepackage[alphabetic]{amsrefs}

\theoremstyle{definition}

\newtheorem*{theorem*}{Theorem}

\numberwithin{equation}{section}

\numberwithin{equation}{section}

\newcommand\op{\operatorname}

\newcommand\mvv\Bbb

\newcommand\pone{\Bbb{P}^1}

\newcommand\tensor{\otimes}
\newcommand\ml{\mathcal{L}}

\newcommand\im{\operatorname{im}}

\newcommand\rk{\operatorname{rk}}

\newcommand\Gr{\operatorname{Gr}}

\newcommand{\leto}[1]{\stackrel{#1}{\to}}

\newcommand\mc{\mathcal}

\newcommand{\ovop}[1]{\overline{\operatorname{#1}}}
\newcommand{\ovmc}[1]{\overline{\mathcal{#1}}}

\newtheorem{theorem}{Theorem}[section]

\newtheorem{remark}[theorem]{ Remark}

\newtheorem{conjecture}[theorem]{Conjecture}

\newtheorem{corollary}[theorem]{Corollary}
\newtheorem{question}[theorem]{Question}

\newtheorem{proposition}[theorem]{Proposition}

\newtheorem{lemma}[theorem]{Lemma}

\theoremstyle{definition}
\newtheorem{definition}[theorem]{Definition}
\newtheorem{definition/lemma}[theorem]{Definition/Lemma}

\newcommand{\sL}{\mathfrak{sl}}

\newtheorem{example}[theorem]{\bf Example}

\begin{document}





\title{Scaling of conformal blocks and generalized theta functions over $\ovmc{M}_{g,n}$}

\author{Prakash Belkale, Angela Gibney, and Anna Kazanova}

\date{}
\maketitle




\begin{abstract} By way of intersection theory on $\ovmc{M}_{g,n}$, we show that geometric interpretations for conformal blocks, as sections of ample line bundles over projective varieties, do not have to hold at points on the boundary. We show such a translation would imply certain recursion relations for first Chern classes
of these bundles.   While recursions can fail, geometric interpretations are shown to hold under certain conditions.

\end{abstract}

\section{Introduction}
Conformal blocks are vector spaces associated to stable curves together with certain Lie theoretic data. These form vector bundles on the moduli stacks
$\overline{\mc{M}}_{g,n}$ parametrizing stable $n$-pointed curves of genus g.  For such a nontrivial bundle $\mathbb{V}$, if $x$ is a smooth n-pointed curve, then by \cites{BeauvilleLaszlo,  Faltings, KNR,  Pauly, LaszloSorger}, there is a  canonical isomorphism $\mathbb{V}|_x^* \cong \op{H}^0(\mc{X}_x,\mc{L}_x)$,
where  $\mc{X}_x$ is a moduli space determined by the Lie data, and $\mc{L}_x$  is a canonical ample line bundle on it.
The  $\op{H}^0(\mc{X}_x,\mc{L}_x)$ are called generalized theta functions.

These isomorphisms at {\em{interior points}} $x$ respect both multiplication operations on global sections and the algebra structure on conformal blocks.  For example, in type $\op{A}$, for $x\in \mc{M}_{g,n}$,
   one has a natural identification  of  graded algebras
 \begin{equation}\label{isomI}
\bigoplus_{m\geq 0} \Bbb{V}[m]|^*_{x}\cong \bigoplus_{m\ge 0} \op{H}^0(\mc{X}_x,\mc{L}_x^{\tensor m}),
 \end{equation}
where the $\Bbb{V}[m]$, for $m \in \mathbb{N}$,  are bundles obtained from $\mathbb{V}$ under an operation called {\em{scaling}} (Def~\ref{Stretching}).

Since these results were proved, it has been an open question as to
whether such  canonical isomorphisms exist at {\em{all points}} $x$ in the moduli space. This is natural, and also interesting, as
each side of Equation \eqref{isomI} has an advantage over the other.  Conformal blocks  give vector bundles
on the entire moduli stack $\ovmc{M}_{g,n}$, and can be used to study its geometry. For instance, first Chern classes are base point free divisors on $\ovmc{M}_{0,n}$ \cite{Fakh}.  On the other hand, for a polarized projective variety $(\mc{X}, \mc{L})$, the algebra $\bigoplus_{m\geq 0}\op{H}^0(\mc{X},\mc{L}^{\tensor m})$ is finitely generated.  Existence of isomorphisms \eqref{isomI} at all boundary points $x$ would imply that the algebra of conformal blocks $\bigoplus_{m\geq 0} \Bbb{V}[m]|^*_{x}$ is finitely generated as well,  an open question.

\smallskip
Our main result is the following:

\begin{theorem}\label{MainI}There are stable $n$-pointed curves $x=(C; p_1,\ldots,p_n)$ and vector bundles of conformal blocks $\mathbb{V}$, for which there are no polarized pairs $(\mc{X}_x,\mc{L}_x)$ such that Equation \ref{isomI} holds.\end{theorem}

To prove Thm \ref{MainI}, we give an obstruction to the existence of a geometric interpretation of conformal blocks at points $x$ on the boundary of $\ovmc{M}_{g,n}$ for  bundles  for which
$(\mc{X}_x,\mc{L}_x)$  are well understood. The $\Delta$-invariant of $(\mc{X}_x,\mc{L}_x)$ is the quantity
$\Delta(\mc{X}_x,\mc{L}_x)= \op{dim}(\mc{X}_x) +\mc{L}^{\op{dim(\mc{X}_x)}} - {h}^0(\mc{X},\mc{L}_x)$ (see Def \ref{DeltaGenus}).   We say that $\mathbb{V}$ has $\Delta$-invariant zero rank scaling if $\Delta(\mc{X}_x,\mc{L}_x)=0$ for $x\in \mc{M}_{g,n}$ such that $\mathbb{V}|_x^*\cong {H}^0(\mc{X}_x,\mc{L}_x)$. Such pairs can be realized as projective varieties of minimal degree (see Section \ref{Delta}).  We show the following.

\begin{theorem}\label{MainLemmaI} Suppose that $\mathbb{V}$ has  $\Delta$-invariant zero rank scaling, and there are polarized pairs $(\mc{X}_x,\mc{L}_x)$ such that Equation \ref{isomI} holds for all points $x$ in $\ovmc{M}_{g,n}$.  Then for $m\ge 1$:
 \begin{equation}\label{MainEqI}
c_1(\Bbb{V}[m])=  \sum_{1\le i \le D} \alpha_i(m) \ c_1(\mathbb{V}[i]), \ \ \mbox{where } \ \ D=\mc{L}_x^{\op{dim}\mc{X}_x}
\end{equation}
and where the coefficients $\alpha_i(m)$ are polynomials in $m$ (see Theorem \ref{MainLemma} for a precise statement).
\end{theorem}

In Example \ref{GoodBad}, and its generalization \ref{ForwardReference},  we give bundles  $\mathbb{V}$ for which Equation (\ref{MainEqI})  fails. Here we use formulas of \cite{Fakh} which compute the Chern classes of $\Bbb{V}$.

The remainder of the paper is devoted to supporting a conjectural description for  the locus $Z$ on $\ovmc{M}_{g,n}$ for which there are always geometric interpretations for conformal blocks  at every point $x \in Z$. Based on the many examples for $g\le 3$, including bundles that do not have $\Delta$-invariant zero rank scaling, we  make following assertion.

\begin{conjecture}\label{ExtendZ}
Given a vector bundle of conformal blocks $\mathbb{V}$ of type $\op{A}$ on $\overline{\mc{M}}_{g,n}$,  there are polarized pairs $(\mc{X}_x,\mc{L}_x)$ such that Equation \ref{isomI} holds for all $x \in Z= \mc{M}^{rt}_{g,n} \cup \Delta^0_{irr}$.  \end{conjecture}
 In the statement of Conjecture \ref{ExtendZ}, the notation $\mc{M}^{rt}_{g,n}$ stands for the set of points in $\ovmc{M}_{g,n}$ corresponding to stable curves in the locus of rational tails.  These are curves with one irreducible component having genus $g$. By $\Delta^0_{irr}$ we mean the set of curves in the interior of the component of the boundary whose generic point has a non-separating node.  We note that for $g=0$, the locus $Z$ is equal to $\ovmc{M}_{0,n}$.

  In Section \ref{HigherGenusEvidence} we illustrate how our techniques are not limited to bundles of conformal blocks with $\Delta$-invariant rank zero, and can be applied as long as one has sufficient information about  polarized pairs $(\mc{X}_x, \mathcal{L}_x)$ at points $x$ on the interior $\mc{M}_{g,n}$. First, we treat the case where $\mc{X}_x$ is embedded by $\mc{L}_x$ as Coble's quartic hypersurface $Z_4$ in $\mathbb{P}^7$.  Second, we consider a cubic hypersurface also studied by Coble: $\mathcal{Z}_3 \hookrightarrow \mathbb{P}^8=\mathbb{P}(\op{H}^0(J^{g-1}C, \mathcal{O}_{JC}(3\theta))$,  for $C$ a smooth curve of genus $2$.  It is known
that  $\mathbb{V}(\sL_3,1)|^*_{[C]}\cong \op{H}^0(SU_{C}(3), \mathcal{L})$, where $SU_{C}(3)$ is a degree $2$ cover of  $\Bbb{P}^8$ branched over a sextic  dual to $\mathcal{Z}_3$ \cite{Ortega,nguyen}.  Finally, we consider an intersection of two quadrics in $\mathbb{P}^5$ using the explicit formulas for first Chern classes in genus $>1$ given in \cite{MOP}.


\subsection{History}\label{History}
Our results can be seen to originate from the Verlinde formula, which gives a closed expression for the dimension of spaces of generalized theta functions.
Although there are finite dimensional proofs in special cases, see \cites{Ber2, Thaddeus, Zagier}, the theory of conformal blocks plays an
 essential role in the general result.   We recommend the survey article of \cite{sorger} for a good account.

While  admittedly  leaving out important contributions, one can summarize the two crucial aspects:
\begin{enumerate}
\item There is a canonical identification between conformal blocks
and generalized theta functions  over smooth curves \cites{BeauvilleLaszlo,  Faltings, KNR,  Pauly, LaszloSorger}.
\item The factorization theorem of Tsuchiya-Ueno-Yamada (see Theorem \ref{Factorization}) allows for a decomposition of conformal blocks at stable pointed curves
$x \in \ovmc{M}_{g,n} \setminus \mc{M}_{g,n}$.
\end{enumerate}

Several interesting questions arising from the above picture have been investigated. One theme, originating in \cites{NR2,DW,Ramadas},
is to try to remove the reference to conformal blocks.  The idea is that if the space of generalized theta functions could be factorized geometrically, then one would obtain a ``finite dimensional" proof of the Verlinde formula. The goal is then to
find pairs $(\mc{X}_x,\mathcal{L}_x)$ for  points $x \in \ovmc{M}_{g,n} \setminus \mc{M}_{g,n}$ such that global sections  have suitable factorization properties. This motif has been continued in the work of Kausz and Sun \cite{Su1,Su2,Kau}. To the best of our knowledge, these factorizations of generalized theta functions have not been related to conformal blocks, also see  \cite{Faltings2,Teleman}.

Our outlook is to suppose that  there is an extension of the identification of conformal blocks and generalized thetas at points on  the boundary, and to ask what consequences may follow.  In type A, there are usually no descent problems, and so it seems reasonable to ask for an interpretation over moduli spaces rather than stacks.    Projective varieties of minimal degree have the advantage that they will degenerate to other projective varieties of minimal degree.  The unique resolutions associated to their ideal sheaves is our main tool. We also find, in special cases, geometric interpretations over the boundary by studying factorization properties of  conformal blocks, and their algebras \cites{TUY, Manon}.

\begin{remark}
There has been recent work on realizing the first Chern classes of conformal blocks over $\overline{M}_{0,n}$ (at a fixed level) geometrically as pull backs of classes under natural morphisms \cite{Fakh, Gi, AGS, Max, GiansiracusaGibney,GJMS,BoG}. The problem considered here in our paper, precisely stated in Question \ref{TheBigOne}, is about realizing the conformal blocks themselves geometrically as sections of ample line bundles on polarized varieties (for a fixed curve).  These problems are distinct from one another.
\end{remark}



\bigskip

\textit{Acknowledgments.}
P.B. was supported on NSF grant  DMS-0901249, and A.G. on  DMS-1201268 and  in part by  DMS-1344994 (RTG in Algebra, Algebraic Geometry, and Number Theory, at UGA).
We thank the anonymous referee for their time and thoughtful feedback.

 \section{Notation and basic definitions}
 \subsection{Conformal blocks}
For  a positive integer  $\ell$ (called the level), we let ${P}_{\ell}(\sL_{r+1})$ denote the set of dominant integral weights $\lambda$ with $(\lambda,\theta)\leq \ell$.  Here $\theta$ is the highest root, and $(\ ,\ )$ is the Killing form, normalized so that $(\theta,\theta)=2$.  To a triple $(\sL_{r+1},\vec{\lambda},\ell)$, such that $\vec{\lambda} \in P_{\ell}(\sL_{r+1})^n$, there  corresponds a vector bundle
$\mathbb{V}$ of conformal blocks  on the stack $\overline{\mathcal{M}}_{g,n}$  \cite{TUY, Fakh}. Throughout the paper, as our notation is fixed, we usually just refer to such bundles as $\mathbb{V}$.
We recommend the Bourbaki article of Sorger ~\cite{sorger} and the survey article of Beauville \cite{Beauville} for some of the background on conformal blocks.  By writing $\mathbb{V}|^*_x \cong {H}^0(\mc{X}_x,\mc{L}_x)$,  for  $x\in \mc{M}_{g,n}$, we mean to indicate that $\mc{X}_x= \mc{X}_x(\sL_{r+1},\vec{\lambda},\ell)$
and $\mathcal{L}_x= \mathcal{L}_x(\sL_{r+1},\vec{\lambda},\ell)$.

\begin{definition}Given a weight $\mu \in \mathcal{P}_{\ell}(\mathfrak{g})$, let $V_{\mu}$ denote the corresponding irreducible representation of $\mathfrak{g}$.  By $\mu^{\star} \in \mathcal{P}_{\ell}(\mathfrak{g})$\index{$\mu^{\star}$} we mean the highest weight in   $(V_{\mu})^*$.
\end{definition}

\begin{theorem}\cite[Factorization]{TUY}\label{Factorization} \ Let $(C_0; \vec{p})$ be a stable $n$-pointed curve of genus $g$ where $C_0$ has a node $x_0$.
\begin{enumerate}
\item If $x_o$ is a non-separating node, $\nu: C\to C_0$  the normalization of $C_0$ at $x_0$, and $\nu^{-1}(x_0)=\{x_{1},x_2\}$,  then
$$\mathbb{V}(\mathfrak{g}, \vec{\lambda}, \ell)|_{(C_0; \vec{p})} \cong \bigoplus_{\mu \in \mathcal{P}_{\ell}(\mathfrak{g})} \mathbb{V}(\mathfrak{g}, \vec{\lambda}\cup \mu \cup \mu^{\star}\}, \ell)_{(C; \vec{p}\cup \{x_1,x_2\})}.$$
\item If   $x_0$ is separating, $\nu: C_1 \cup C_2 \to C_0$ the normalization at $x_0$,  $\nu^{-1}(x_0)=\{x_{1},x_2\}$, with $x_i \in C_i$, then
$$\mathbb{V}(\mathfrak{g}, \vec{\lambda}, \ell)|_{(C_0; \vec{p})}
\cong \bigoplus_{\mu \in \mathcal{P}_{\ell}(\mathfrak{g})} \mathbb{V}(\mathfrak{g}, \lambda(C_1) \cup \{\mu\}, \ell)|_{(C_1; \{p_i \in C_1\}\cup \{x_1\})}\tensor \mathbb{V}(\mathfrak{g},\lambda(C_2)\cup \{\mu^{\star}\}, \ell)|_{(C_2; \{p_i \in C_2\}\cup \{x_2\})},$$
\noindent
where $\lambda(C_i)=\{\lambda_j | p_j \in C_i\}$.
\end{enumerate}
\end{theorem}

\begin{definition}\label{RestrictionData}
The collection of weights $\{(\mu,\mu^*) : \mu \in \mathcal{P}_{\ell}(\mathfrak{g})\}$ which appear in Theorem \ref{Factorization}, and which give nonzero summands are called the {\textbf{restriction data}} for a bundle $\mathbb{V}(\mathfrak{g}, \vec{\lambda}, \ell)$
at $(C_0; \vec{p})$.
\end{definition}

 In many of the examples and computations done here, we find the ranks of conformal blocks bundles  using the  cohomological form of Witten's Dictionary, which expresses these ranks as the intersection numbers  of particular classes (depending on the bundle) in the small quantum cohomology ring of certain Grassmannian varieties.  

\begin{theorem}\cite[Cohomological version of Witten's Dictionary]{b4}\label{WD} \ For $\sum |\lambda_i| =(r+1)(\ell +s)$.
\begin{itemize}
\item If $s<0$, then the rank of $\mathbb{V}(\sL_{r+1}, \vec{\lambda}, \ell)$ coincides with the rank of classical coinvariants $\mathbb{A}(\sL_{r+1}, \vec{\lambda})$.
\item If $s\geq 0$, let $\lambda=(\ell,0,\dots,0)$. The rank of $\mathbb{V}(\sL_{r+1}, \vec{\lambda}, \ell)$ is the coefficient of $q^{s}[\operatorname{pt}]=q^s\sigma_{(\ell,\ell,\dots,\ell)}$  in
$$\sigma_{\lambda_1}\star\dots\star\sigma_{\lambda_n}\star\sigma_{\lambda}^{s}\in QH^*(Y),\  Y=\Gr (r+1, r+1+\ell),$$
where  $\sigma_{\lambda}^{s}$ is the $s$-fold quantum $\star$ product of $\sigma_{\lambda}$.
One can write the above multiplicity also as the coefficient of $q^s\sigma_{\lambda^c_n}$ ($\lambda^c_n$ is the complement of $\lambda_n$ in a $(r+1)\times \ell$ box) in
$$\sigma_{\lambda_1}\star\sigma_{\lambda_2}\star\dots\star\sigma_{\lambda_{n-1}}\star\sigma_{\lambda}^{s}\in QH^*(Y).$$
\end{itemize}
\end{theorem}

  The relation to quantum cohomology follows from \cite{witten} and the twisting procedure of \cite{b4}, see Eq (3.10) from \cite{b4}.
Examples of such rank computations were done using Witten's Dictionary in \cites{BGMB, Kaz, BGMA}, and \cite{Hobson}.

\subsection{$\Bbb{V}[m]$ and the algebra of conformal blocks}\label{ManonAlgebra}

\begin{definition}\label{Stretching}
For $\lambda_i=\sum_{j=1}^rc_j \omega_j  \in P_{\ell}(\sL_{r+1})$, and $m\in \mathbb{N}$, set $m\lambda_i=\sum_{j=1}^r (m  c_j ) \omega_j \in P_{m \ell}(\sL_{r+1})$.
 Given $\mathbb{V}=\mathbb{V}(\sL_{r+1},\vec{\lambda},\ell)$, set
$$\Bbb{V}[m]=\mathbb{V}(\sL_{r+1},m\vec{\lambda}, m \ell),$$ where $m\vec{\lambda}=(m\lambda_1,\ldots,m\lambda_n) \in P_{m\ell}(\sL_{r+1})^n$.
We often refer to the new bundles $\Bbb{V}[m]$ as {\bf{multiples}} of $\mathbb{V}$, and we say they are obtained by {\bf{scaling}} the Lie data used to form $\mathbb{V}$.
\end{definition}

Using the $\Bbb{V}[m]$, one can form a flat sheaf of algebras (see \cite[p.~368]{Faltings}, \cite{Manon}):
\begin{equation}\label{Algebra}
 \mathcal{A}=\bigoplus_{m \in \mathbb{Z}_{\ge 0}}\mathcal{A}_m=\bigoplus_{m\geq 0}\  \Bbb{V}[m]^*,
 \end{equation}
over $\ovmc{M}_{g,n}$.    At any {\em{interior}} point $x\in \mc{M}_{g,n}$,  one has a natural identification  of  graded algebras
 \begin{equation}\label{isom}
 \mathcal{A}|_x =\bigoplus_{m \in \mathbb{Z}_{\ge 0}}(\mathcal{A}_m)|_x=\bigoplus_{m\geq 0}\  \Bbb{V}[m]^*|_{x}\cong \bigoplus_{m\ge 0} {H}^0(\mc{X}_x,\mc{L}_x^{\tensor m}).
 \end{equation}
Here, for $x=(C,p_1,\dots,p_n)\in \mc{M}_{g,n}$, one has that $\mc{X}_x= \mc{X}_x(\sL_{r+1},\vec{\lambda},\ell)$ is a moduli space, constructed using GIT, of  parabolic bundles of rank $r+1$ with trivial determinant on the curve $C$ with parabolic structures at $p_1,\dots,p_n$, and $\mathcal{L}_x= \mathcal{L}_x(\sL_{r+1},\vec{\lambda},\ell)$ an ample line bundle.

We ask if a similar description for  the algebra of conformal blocks holds at points $x\in \ovmc{M}_{g,n} \setminus \mc{M}_{g,n}$:

\begin{question}\label{TheBigOne}Given a vector bundle of conformal blocks $\mathbb{V}$ on $\ovmc{M}_{g,n}$ in type A,
 and $x\in \ovmc{M}_{g,n} \setminus \mc{M}_{g,n}$ is there a polarized scheme $(\mc{X}_x,\mathcal{L}_x)$ such that
Eq  \eqref{isom} holds as graded algebras?
 \end{question}

 \begin{definition}\label{GeometricInterpretation}
Given a vector bundle of conformal blocks $\mathbb{V}$ on $\ovmc{M}_{g,n}$, if the answer to Question \ref{TheBigOne} is yes for $\mathbb{V}$, then we will say that $\mathbb{V}$ has {\textbf{geometric interpretations}} at boundary points $x \in \ovmc{M}_{g,n}\setminus \mc{M}_{g,n}$.
\end{definition}

\begin{remark}\label{integro}
\begin{enumerate}
\item  We note that $\mathcal{A}_x$ for  $x\in \ovmc{M}_{g,n}$ is an integral domain, since it  is a subalgebra (formed by suitable Lie-algebra invariants) of  the algebra of sections of a line bundle on the ind-integral affine Grassmannian times an $n$-fold product of complete flag varieties (see e.g., \cite[Section 10]{LaszloSorger}). Therefore, the $\mc{X}_x$ are necessarily irreducible varieties when they exist.
\item
 The algebras $\mathcal{A}_x$ are finitely generated for $x\in \mc{M}_{g,n}$, since the coordinate ring of a polarized variety is finitely generated.  We do not know if  the $\mathcal{A}_x$ are finitely generated for $x\in \ovmc{M}_{g,n}\setminus \mc{M}_{g,n}$.
 If $\mathcal{A}_x$ is finitely generated for $x\in \ovmc{M}_{g,n}\setminus \mc{M}_{g,n}$ then, $\mc{X}_x$ (if it exists) will coincide with $\operatorname{Proj}(\mathcal{A}_x)$. However $\operatorname{Proj}(\mathcal{A}_x)$ need not carry an ample  line bundle $\mathcal{L}_x$ whose section ring equals $\mathcal{A}_x$.
 \end{enumerate}
\end{remark}

Note that by Remark \ref{integro}, such $\mc{X}_x$ (if they exist) are necessarily reduced and irreducible.  Moreover, in this question we do not require that $(\mc{X}_x,\mathcal{L}_x)$ fit together into a flat family.
Question \ref{TheBigOne} is a  point-wise version of the  question of existence of a family $\pi:\mathcal{X}\to\ovmc{M}_{g,n}$ with relatively ample bundles $\ml$ such that $ \pi_*\ml^{\otimes m}=\Bbb{V}[m]^*$ consistent with multiplication operations. One could then hope to (recursively) control the Chern classes of $\Bbb{V}[m]$ by applications of
the Grothendieck-Riemann-Roch formula. In the situation when $\Delta(\mc{X}_x,\mathcal{L}_x)=0$ for $x\in \mc{M}_{g,n}$ (described next in Def \ref{DeltaGenus}), the point-wise question is implied by the family question by cohomology vanishing (\cite[Chapter 1, (5.1)]{fujita} and  Remark \ref{vanish}).

\subsection{The \texorpdfstring{$\Delta$}{Delta}--invariant and projective varieties of minimal degree}\label{Delta}

\begin{definition}\label{DeltaGenus}  Let $X$ be an  irreducible projective variety, and $L$ an ample line bundle on $X$.
\begin{enumerate}
\item The self intersection  $L^{\op{dim}(X)}$ is $\op{dim}(X)!$ times the leading coefficient of the Hilbert polynomial
$f(m)=\chi(X,L^{\tensor m})$ (see \cite[3.11]{fujita}, for example). The self intersection $D=L^{\op{dim(X)}}$ is also denoted  called the degree of the polarized variety $(X,L)$.

\item The $\Delta$-invariant of the polarized variety $(X,L)$ is given by the formula
$$\Delta(X,L)= \op{dim}(X) +L^{\op{dim(X)}} - {h}^0(X,L).$$
We refer to $\Delta(X,L)$ as the {\bf{$\Delta$-genus}} or as the {\bf{$\Delta$-invariant}} of the pair $(X,L)$.
\end{enumerate}
\end{definition}

We note that by \cite[Chapter 1, (4.2), (4.12)]{fujita}, and  \cite[ Theorem 3.1.1]{BeltraSom}, $\Delta(X,L)\geq 0$, and if $\Delta(X,L)=0$, then $\op{H}^0(X,L)$ generates the algebra $\bigoplus_{m\geq0} {H}^0(X,L^{\otimes m})$ and hence $L$ is very ample, giving rise to an embedding  into a projective space
$$X\hookrightarrow \op{Proj}(\op{B}_{\bullet})= \mathbb{P}^N, \ \ \op{B}_{\bullet}=\bigoplus_{m \in \mathbb{Z}_{\ge 0}} \op{Sym}^m({H}^0(X,L)),$$
so that its image is a (nondegenerate) variety of degree equal to $L^{\op{dim}(X)}=1+\op{codim}(X)$.

\begin{definition}
A nondegenerate variety $\op{X}\subset \mathbb{P}^N$ always has $\op{deg}(\op{X})\ge 1+\op{codim}(\op{X})$, and is said to be  {\bf{of minimal degree}} if $\op{deg}(\op{X})=1+\op{codim}(\op{X})$.
\end{definition}

Therefore, polarized varieties $(X,L)$ with
$\Delta(X,L)=0$ correspond to projective varieties of minimal degree.  A description of these varieties has been given by many authors including the classification of minimal surfaces  \cite{DP} (also \cite{Nagata}),  and the higher  dimensional cases by Bertini in  1907 \cite{Bertini}, and subsequent treatments by Harris, Xamb{\'o}, Griffiths and Harris, Fujita, and Beltrametti and Sommese \cite{HarrisBound, X, GH, fujita,BeltraSom}.  We recommend  \cite{EisHarrMinDeg} for a good historical account.

\begin{proposition}\label{classy} \cite[Chapter 1, (5.10), (5.15)]{fujita},  \cite[Proposition 3.1.2]{BeltraSom} We suppose that  $\Delta(X,L)=0$ and $d=\dim X \ge 2$:
\begin{enumerate}
\item $(X,L)\cong (\Bbb{P}^d,\mathcal{O}_{\Bbb{P}^d}(1))$ if $L^d=1$;
\item  $(X,L)\cong (Q,\mathcal{O}_{Q}(1))$, where $Q$ is a not necessarily smooth quadric in $\Bbb{P}^{d+1}$ if $L^d=2$;
\item  $(X,L)$ is a $\mathbb{P}^{d-1}$ bundle over $\mathbb{P}^1$, $X\cong \mathbb{P}(\mc{E})$, for a vector bundle $\mc{E}$ on $\mathbb{P}^1$ which is a direct sum of line bundles of positive degrees;
\item  $(X,L)\cong (\Bbb{P}^2,\mathcal{O}_{\Bbb{P}^2}(2))$; or
\item  $(X,L)$ is a (normalized) generalized cone over a smooth subvariety $V \subset X$ with $\Delta(V,L_{V})=0$, where $L_{V}$ denotes the restriction of $L$ to $V$.
\end{enumerate}
\end{proposition}
\noindent

For (1) in Proposition \ref{classy}, we note the related paper \cite{goren}.

\begin{definition}\label{DRankScaling}
We say that $\mathbb{V}$ has {\bf $\Delta$-invariant zero rank scaling} if $\Delta(\mc{X}_x,\mc{L}_x)=0$ for $x\in \mc{M}_{g,n}$ such that $\mathbb{V}|_x^*\cong {H}^0(\mc{X}_x,\mc{L}_x)$.  \end{definition}

\begin{remark}
The function $f(m)=\rk \Bbb{V}[m]$ determines $\Delta(\mc{X}_x,\mc{L}_x)$:
\begin{enumerate}
\item $f(1)=h^0(\mc{X}_x,\mc{L}_x)$;
\item  $\op{dim}(\mc{X}_x)=\op{deg}(f(m))$; and
\item $\mc{L}_x^{\op{dim}\mc{X}_x}=c \times \dim \mc{X}_x!$, where $c$ is the leading coefficient of $f(m)$.
\end{enumerate}
Moreover, $H^i(\mathcal{X}_x,\mc{L}_x^{\tensor m})=0$  for $i>0, m\geq 0$ by  \cite[Theorem 9.6]{telly}, so the Hilbert polynomial is determined by the first few values of $\rk \Bbb{V}[m]$. The dimension of $\mc{X}_x$  can be bounded above, in genus $0$ it is no more than the dimension of the corresponding flag variety. There is also an explicit formula for $\mathcal{L}_x^{\op{dim}\mc{X}_x}$,  the degree of $\mathcal{L}_x$ due to Witten \cite{Wi} (also see Sections 3,\ 4  in \cite{BBV} for examples).

\end{remark}

\begin{definition} We refer to $\mc{L}_x^{\dim \mc{X}_x}$ as the {\textbf{degree of the block}} $\mathbb{V}$. 
 \end{definition}

\begin{example}If $X$ has dimension one and  $\Delta(X,L)=0$ then $(X,L)=(\pone,\mathcal{O}(D))$
for some $D>0$. Indeed, $\Delta(\pone,\mathcal{O}(D))=0$. Conversely, $0=\Delta(X,L) = 1+\op{deg}(L) - {h}^0(L)$. Riemann-Roch gives that $\chi(X,L)= 1-g_a(X) + \op{deg}(L)$.
Hence ${h}^1(X,L) = -g_{a}(X)=0$, where $g_a(X)$ is the arithmetic genus of  $X$.
\end{example}

\begin{definition}\label{TOS}We refer to the various types of $\Delta$-invariant zero rank scaling according to the classification of projective varieties of minimal degree described in Proposition \ref{classy}.  For example:
\begin{enumerate}
\item $\op{rk}(\mathbb{V}[m])={{d+m}\choose m}$, then $\mathbb{V}$ has projective rank scaling;
\item $\op{rk}(\mathbb{V}[m])=2{{m+d-1}\choose {d}} + {{m+d-1} \choose {d-1}}$, then $\mathbb{V}$ has quadric rank scaling;
\item $\op{rk}(\mathbb{V}[m])=(m+1)(1 + \frac{m(a+b)}{2})$, then $\mathbb{V}$ has  $(S(a,b),\mathcal{O}(1))$ scaling\footnote{One can write down a more general sequence for the scrolls $(S(a_1,\dots,a_d)=\Bbb{P}(\mathcal{E}), \mc{O}(1))$, discussed in Section \ref{scrolls}.};
\item $\op{rk}(\mathbb{V}[m])=(m+1)(2m+1)$, then $\mathbb{V}$ has Veronese surface scaling; and
\item $\op{rk}(\mathbb{V}[m])=dm+1$, then $\mathbb{V}$ has $(\Bbb{P}^1,\mathcal{O}(d))$, or rational normal curve scaling.
\end{enumerate}
\end{definition}

\section{Divisor class formulas in case extensions exist}\label{mimicry}
Here we prove the following result, which is the full statement of Theorem \ref{MainLemmaI}, from the introduction.
\begin{theorem}\label{MainLemma} Suppose that $\mathbb{V}$ has  $\Delta$-invariant zero rank scaling, and assume that geometric interpretations exist for $\mathbb{V}$ at all points (see Definition \ref{GeometricInterpretation}).  Let
$D=\deg(\mc{L}_x)$ for $x\in \mc{M}_{g,n}$, and $R=\rk \Bbb{V}$. Then there exist vector bundles $\mathcal{W}_i, i=1,\dots,D$ on $\ovmc{M}_{g,n}$ such that for all integers $m\geq 1$,
\begin{multline}\label{MainEq}
c_1(\Bbb{V}[m])=\Bigg({{m+R-1}\choose{R}}  + \sum_{i=2}^{D}(-1)^{i-1}(i-1){{D}\choose{i}}{{m-i+R-1}\choose{R}}\Bigg)c_1(\mathbb{V})+\sum_{i=2}^D(-1)^i{{m-i+R-1}\choose{R-1}} c_1(\mathcal{W}_i).
\end{multline}
The first Chern class $c_1(\mathcal{W}_i)$ can be written in terms of $c_1(\Bbb{V}[m]),m=1,\dots,i$ using Equation \eqref{MainEq}. Therefore Equation \eqref{MainEq} expresses $c_1(\Bbb{V}[m])$ in terms of $c_1(\Bbb{V}[1]),\dots, c_1(\Bbb{V}[D])$ for all $m$.
\end{theorem}

\begin{remark}
\begin{enumerate}
\item Crucially,
if  $\mathbb{V}$ has  $\Delta$-invariant zero rank scaling at some  $x$, and if geometric interpretations exist for $\mathbb{V}$ at all points (see Definition \ref{GeometricInterpretation}), then
 $\Delta(\mc{X}_x,\mc{L}_x)=0$ for all points $x\in \ovmc{M}_{g,n}$. 
 \item More generally, in any family of polarized varieties, the 
 $\Delta$-invariant is upper semicontinuous (i.e., can only go down at special points).
 \end{enumerate}
\end{remark}

\subsection{Outline of proof}
The short alternate proof in the case of quadrics in Corollary \ref{QuadricIdentity} contains all the essential ideas. The proof of Theorem \ref{MainLemma} proceeds in two steps.
\begin{enumerate}
\item Suppose $(X,L)$ is a polarized variety of $\Delta$-invariant zero. Then we have a canonical  embedding of $X$ in a projective space
$X\hookrightarrow \op{Proj}(\op{B}_{\bullet})$ with $\op{B}_{\bullet}=\bigoplus_{m \in \mathbb{Z}_{\ge 0}} \op{Sym}^m({H}^0(X,L))$.
In Section \ref{sho}, we recall the canonical resolution of the ideal sheaf  $\mathcal{I}_{\op{X}}$.
\item In Section \ref{recollement}, we glue these resolutions of the ideal sheaves corresponding to $(\mathcal{X}_x,\mc{L}_x)$ for all $x\in \ovmc{M}_{g,n}$. This will result in an
exact sequence \eqref{LESsheaves} which yields  \eqref{MainEq}, using Remark \ref{rankc1}.
\end{enumerate}

\begin{remark}\label{rankc1}
If $V$ is a vector bundle of rank $R$ then
$\rk (\operatorname{Sym}^m V)= {{m+R-1}\choose{R-1}}$, and $c_1 (\operatorname{Sym}^m V)= {{m+R-1}\choose{R}} c_1(V)$.
\end{remark}

\subsection{Resolutions of ideals of projective varieties of minimal degree}\label{sho}

If the pair $({X},{L})$ is a polarized variety having $\Delta$-invariant zero, then ${L}$ is very ample on ${X}$,  then from ${L}$ one obtains an embedding
$X\hookrightarrow \op{Proj}(\op{B}_{\bullet})$, as a variety of minimal degree.  In particular,
 the ideal $\mathcal{I}_{\op{X}}$, of $X$
  is the kernel of the surjective map
$$\phi:\op{B}_{\bullet}=\bigoplus_{m \in \mathbb{Z}_{\ge 0}} \op{Sym}^m({H}^0(X,L)) \twoheadrightarrow \op{A}_{\bullet}=\bigoplus_{m \in \mathbb{Z}_{\ge 0}} {H}^0(X,L^{\otimes m}),$$
and $\mathcal{I}_{\op{X}}$ is  known to have a minimal free resolution of the form
  \begin{equation}\label{mimic}
  0\to W_D\tensor\mathcal{O}(-D)\to\dots  \to W_3\tensor\mathcal{O}(-3)\to  W_2\tensor\mathcal{O}(-2)  \to \mathcal{I}_{\op{X}}\to 0
  \end{equation}
where for $i \in \{1,\ldots, D\}$, the  $W_i$ are vector spaces of dimension $(i-1){{D}\choose{i}}$, and $D=L^{dim{\op{X}}}$ is the degree of $L$  \cite{UWE,EisenbudGoto}.

\begin{remark}\label{vanish}
Let $\mathcal{J}_k$ be the image of the map $W_{k+1}\tensor\mathcal{O}(-k-1)\to W_{k}\tensor\mathcal{O}(-k)$ in \eqref{mimic}. Then the exactness of \eqref{mimic} implies that
${H}^i(\Bbb{P}^N,\mathcal{J}_k (m))=0$ for all $m\geq 0$ and $i>0$, and for $i=0$ and $m<k+1$. This is proved by induction using the
exact sequences
 $$0\to \mathcal{J}_k\to W_{k}\tensor\mathcal{O}(-{k})\to \mathcal{J}_{k-1}\to 0.$$
 Note that this also implies that ${H}^i(\Bbb{P}^N,\mathcal{I}_X(m))={H}^i(X,L^{\otimes m})=0$ for all  for $i \ge 1$, and $m\geq 0$ and
 $$\op{H}^0(\op{X},\op{L}^{\otimes m})= {H}^0(\Bbb{P}^N,\mathcal{O}(m))/{H}^0(\Bbb{P}^N,\mathcal{I}_X(m)).$$
\end{remark}

Minimal free resolutions in the context of graded rings are known to be unique up to unique isomorphisms. The same
 technique shows that the complex \eqref{mimic} is unique up to a unique isomorphism.
  Twisting \eqref{mimic} by $\mathcal{O}(k)$ with $2\leq k\leq D$, and then taking global sections gives an exact sequence (see Remark \ref{vanish})
 \begin{equation}\label{Kernel}
 0\!\rightarrow\! W_k\rightarrow W_{k-1}\tensor\operatorname{Sym}^1(B_1)\!\rightarrow\!\dots\!\rightarrow \!W_2\tensor  \operatorname{Sym}^{k-2}(B_1)\!\rightarrow \!{H}^0(\Bbb{P}^N,\mathcal{I}(k)) \rightarrow 0.
 \end{equation}
 Therefore, $W_k$ can be reconstructed uniquely, as a kernel, by induction.

\begin{lemma}\label{Canonical}Given a pair $(X,L)$, having $\Delta$-invariant zero, with notation as above, there are canonical vector spaces ${W}_{i}$,  for  $2 \le i \le D$,  and a long exact sequence of the form:
 \begin{equation}\label{LESmodules}
 0 \rightarrow {W}_{D} \otimes \op{B}_{\bullet} (-D)\overset{\phi_i}{\rightarrow} \cdots \rightarrow {W}_{3}  \otimes \op{B}_{\bullet}(-3)  \overset{\phi_1}{\rightarrow} {W}_{2}  \otimes \op{B}_{\bullet} (-2)\overset{\phi_0}{\rightarrow} \mc{I}_{\bullet} \rightarrow 0,
 \end{equation}
giving a graded resolution of  the graded ideal $\mc{I}_{\bullet}=\bigoplus_{m\geq 0} \mc{I}_{m}$ with $\mc{I}_{m}= {H}^0(\Bbb{P}^N,\mathcal{I}_X(m))$.  In particular
 \begin{equation}\label{LESmodulesm}
 0 \! \rightarrow\! {W}_{D} \otimes \op{B}_{m-D} (-D)\overset{\phi_i}{\rightarrow}\! \cdots \!\rightarrow \!{W}_{3}  \otimes \op{B}_{m-3}(-3)  \overset{\phi_1}{\rightarrow} \!{W}_{2}  \otimes \op{B}_{m-2} (-2)\overset{\phi_0}{\rightarrow}\! \mc{I}_{m}\! \rightarrow\! 0
 \end{equation}
is a canonical resolution of the degree $m$ part $\mc{I}_{m}$.
  \end{lemma}

\subsection{Resolutions in families}\label{recollement}
Suppose that $\mathbb{V}$ has  $\Delta$-invariant zero rank scaling, and assume that geometric interpretations exist for $\mathbb{V}$ at all points. Then for each $x\in \ovmc{M}_{g,n}$, we have a resolution of the form \ref{Canonical}. In this section we will show that we can glue these resolutions, and get an exact sequence over $\ovmc{M}_{g,n}$ of the form:

\begin{equation}\label{LESsheaves}
0  \rightarrow \mathcal{W}_D \otimes \op{Sym}^{m-D}(\Bbb{V})
\rightarrow \cdots  \rightarrow  \mathcal{W}_2 \otimes \op{Sym}^{m-2}(\Bbb{V})  \rightarrow \op{Sym}^{m}(\Bbb{V})  \overset{\phi_m}{\rightarrow} \Bbb{V}[m]^*  \rightarrow 0,
\end{equation}
where $\mathcal{W}_i$ are vector bundles on $\ovmc{M}_{g,n}$. This will finish the proof of Theorem \ref{MainLemma}.

 For a given $\mathbb{V}$, we form the following sheaves on $\ovmc{M}_{g,n}$: Let $\mc{A}=\bigoplus_{m\ge 0}\mc{A}_m$, where $\mc{A}_m=\Bbb{V}[m]^*$ and $\mc{A}_{0}=\mc{O}$. Let $\mc{B}_{\bullet}=\bigoplus_{m \in \mathbb{Z}_{\ge 0}}\mc{B}_m$, where $\mc{B}_m=\op{Sym}^m\mc{A}_1$. Consider $\phi_m: \mc{B}_{m}= \op{Sym}^{m}(\mc{A}_1)\longrightarrow \mc{A}_{m}, \ \ \mbox{and }  \mc{I}_{m}=\op{ker}(\phi_m).$ Clearly $\mc{I}_{\bullet}$ is a graded $\mc{B}_{\bullet}$ ideal.

\begin{proposition}\label{VB}
There is an exact sequence of the form \eqref{LESsheaves}
with fibers equal to the resolution of the $\op{ker}(\phi_m)|_x$,  given in Lemma \ref{Canonical}.
Furthermore, \eqref{LESsheaves} is exact and $\mathcal{W}_i$ are locally free of finite rank.
\end{proposition}
\begin{proof}
We first build a resolution
\begin{equation}\label{H}
\cdots  \longrightarrow \mathcal{F}_{k} \rightarrow \cdots \rightarrow \mathcal{F}_{k-1}  \rightarrow \cdots \rightarrow \mathcal{F}_{0} \rightarrow \mc{I}_{\bullet}.
 \end{equation}

Let $\mc{I}^0_{\bullet}=\mc{I}_{\bullet}.$ We note that $\mc{I}^0_{\bullet}$  is a graded $\mc{B}_{\bullet}$-module.
Because $\mc{I}^0_{\bullet}$ is an ideal, there is a  morphism $\mc{I}^0_{2} \otimes \mc{B}_{m-2} \overset{\alpha^0_m} \longrightarrow \mc{I}^0_{m}$, and we set $\mc{I}^1_{m}=\op{ker}(\alpha^0_m)$.
Now we consider the morphism $\mc{I}^1_{3} \otimes \mc{B}_{m-3} \overset{\alpha^1_m} \longrightarrow \mc{I}^1_{m}$, and we set $\mc{I}^2_{m}=\op{ker}(\alpha^1_m)$.
At the $i$-th step we consider the morphism $\mc{I}^{i-1}_{i+1} \otimes \mc{B}_{m-(i+1)} \overset{\alpha^{i-1}_m} \longrightarrow \mc{I}^{i-1}_{m}$, and set $\mc{I}^i_{m}=\op{ker}(\alpha^{i-1}_m)$.

The terms in \eqref{H} are defined as follows. The sheaf $\mathcal{F}_{0}$ is $\bigoplus_{m} \mc{I}^0_{2} \otimes \mc{B}_{m-2}$ with $\mc{B}_i=0$ for $i<0$. The sheaf $\mathcal{F}_{1}$  is $\bigoplus_{m}\mc{I}^1_{3} \otimes \mc{B}_{m-3}$
and $\mathcal{F}_{i}$ is $\bigoplus_{m}\mc{I}^i_{i+2} \otimes \mc{B}_{m-i-2}$. Finally set $\mathcal{W}_i=\mc{I}^i_{i+2}$. This is the sheaf that appears in \eqref{LESsheaves}.

The map $\mathcal{F}_{0}\to \mc{I}_{\bullet}$ is clear because $\mc{I}_{\bullet}$ is a graded $\mc{B}_{\bullet}$ module. The sheaf $\mathcal{F}_{1}$  maps to $\mathcal{F}_0$ since $\mc{I}^1_{3}$  sits inside $\mathcal{F}_{0}$. One can define the maps in \eqref{H} by continuing this process. The general process is the following: if $f:M\to N$ is a map of graded modules for a graded ring $R$, then for any positive integer $s$,  $\ker(f_s)\tensor R(-s)\to M\to N$ is a complex (possibly non-exact).

Working on the (reduced) scheme $\{x\}$, we form a complex similar to \eqref{H}.   Let $\mc{A}'_{m}=(\mc{A}_m)\tensor k(x)$.
 \begin{equation}\label{HilbertSfiber}
\cdots  \longrightarrow \mathcal{F}'_{k} \rightarrow \cdots \rightarrow \mathcal{F}'_{k-1}  \rightarrow \cdots \rightarrow \mathcal{F}'_{0} \rightarrow \mc{I}'_{\bullet}.
 \end{equation}

Here $\mc{I}'_{\bullet}$  is $\op{ker}(\phi'_m)$ where  $\phi'_m: \mc{B'}_{m}= \op{Sym}^{m}(\mc{A}'_1)\longrightarrow \mc{A}'_{m}$.  From our assumption on $\Delta$-invariants (and Section \ref{sho}) it follows that \eqref{HilbertSfiber} is exact, and that the maps $\phi'_m$ are surjective. Note that for $k\geq 2$,
$\mc{I}'^{k-2}_{k}$ coincides with ${W}_{k}$ in Lemma \ref{Canonical} (as a suitable kernel from \eqref{Kernel}).
Now $\mc{I}_m$ is a kernel of a surjective map of vector bundles since $\phi_m'$ are surjective, and the map  $\mc{I}_{\bullet}\tensor k(x)\to \mc{I}'_{\bullet}$ an isomorphism.

For the inductive step, fix a point $x$, and suppose
$\psi:\mathcal{G}_{\bullet}\to\mathcal{H}_{\bullet}$ is morphism of graded $\mc{B}_{\bullet}$ modules.
Let $\mathcal{F}_{\bullet}$ be the kernel of $\psi$ and $\mathcal{F}'_{\bullet}$ the kernel of $\psi\tensor k(x)$
and $s\in \Bbb{Z}$. Assume
\begin{itemize}
\item $\mathcal{G}_m$ and $\mathcal{H}_m$  are vector bundles for all $m$;
\item  images of $\psi_m$ are locally free subbundles of $\mathcal{H}_m$ in the localization $\operatorname{Spec}(\mathcal{O}_x)$; and
\item the three term sequence on fibers is exact:  $\mathcal{F}'_s\tensor\mc{B}'_{\bullet}(-s) \to \mathcal{G}_{\bullet}\tensor k(x)\to \mathcal{H}_{\bullet}\tensor k(x)$.
\end{itemize}
Then
$\mathcal{F}_s\tensor\mc{B}_{\bullet}(-s) \to \mathcal{G}_{\bullet}\to \mathcal{H}_{\bullet}$ is exact over $\operatorname{Spec}(\mathcal{O}_x)$, and
hence $\mathcal{F}_{m}$ is the image of $\mathcal{F}_s\tensor\mc{B}_{m-s}$ in  $\mathcal{G}_m$, and is a subbundle
of the latter. The next step in the induction is with $\psi$ the map $\mathcal{F}_s\tensor\mc{B}_{\bullet}[-s] \to \mathcal{G}_{\bullet}$.

We work over $\operatorname{Spec}(\mathcal{O}_x)$.  Note that
$0\to \mathcal{F}_m\to \mathcal{G}_m\to \im(\psi)_m\to 0$
is an exact sequence, and
therefore, $\mathcal{F}_m$ are vector bundles. Now tensoring it by  $k(x)$
we obtain another exact sequence. Note that $\im(\psi)_m\tensor k(x)\to \mathcal{H}_m\tensor k(x)$
is injective since  $\im(\psi)_m\to \mathcal{H}_m$ is a sub-bundle. This tells us
that $\mathcal{F}_m\tensor k(x)\to \mathcal{F}'_m$ is an isomorphism. Finally
$\mathcal{F}_s\tensor\mc{B}_{m-s}\to \mathcal{F}_m$ is a surjection of vector bundles
because it is so after tensoring with $k(x)$ by our third hypothesis.
\end{proof}

\begin{remark}
For the validity of Equation \eqref{MainEq}, we only need geometric interpretations on the generic points of all boundary divisors, which would imply that Equation \eqref{MainEq} holds in $\operatorname{Pic}(\ovmc{M}_{g,n}\setminus Z)$ with $Z$ of codimension $2$. Since Picard groups are unaffected by removal of codimension $2$ substacks,  Equation \eqref{MainEq} would then hold on $\ovmc{M}_{g,n}$.
\end{remark}

\section{Cases when obstructions are non-zero}

In this section we prove Theorem \ref{MainI}, which we restate below:

\begin{theorem*}There are stable $n$-pointed curves $x=(C; p_1,\ldots,p_n)$ and vector bundles of conformal blocks $\mathbb{V}$, for which there are no polarized pairs $(\mc{X}_x,\mc{L}_x)$ such that
$$
\bigoplus_{m\geq 0} \Bbb{V}[m]|^*_{x}\cong \bigoplus_{m\ge 0} \op{H}^0(\mc{X}_x,\mc{L}_x^{\tensor m}),
$$
 holds.\end{theorem*}

\begin{proof}  If such an isomorphism were to exist, then if $\mathbb{V}$ has  $\Delta$-invariant zero rank scaling,  by Thm \ref{MainLemma},
$$c_1(\Bbb{V}[m])=\Bigg({{m+R-1}\choose{R}}  + \sum_{i=2}^{D}(-1)^{i-1}(i-1){{D}\choose{i}}{{m-i+R-1}\choose{R}}\Bigg)c_1(\mathbb{V})+\sum_{i=2}^D(-1)^i{{m-i+R-1}\choose{R-1}} c_1(\mathcal{W}_i).$$
One can determine the $c_1(\mathcal{W}_i)$  in terms of $c_1(\Bbb{V}),\dots, c_1(\Bbb{V}[i])$ by setting  $m=2,\dots, D$ in Equation \eqref{MainEq}. When one substitutes $m=i$ with $2\leq i\leq D$,  then Equation \eqref{MainEq} is of the form $(-1)^i c_1(\mathcal{W}_i)$ plus a linear expression in $c_1(\mathcal{W}_2),\dots,c_1(\mathcal{W}_{i-1})$, and $c_1(\Bbb{V})$.  These identities can be shown to fail, as  in Examples  \ref{GoodBad} and  \ref{ForwardReference}.
\end{proof}

\begin{remark}\label{NotLimited}
 Our techniques are not limited to bundles with $\Delta$-invariant zero rank scaling.  Identities analogous to Equation (\ref{MainEqI}) exist when one has enough knowledge about pairs $(\mc{X}_x,\mc{L}_x)$ at points $x$ on the boundary, if they exist.  To illustrate, we give  identities where $\mc{X}_x$ could be embedded by $\mc{L}_x$ as Coble's quartic hypersurface in $\mathbb{P}^7$ (Proposition \ref{recurso}), Coble's cubic hypersurface in $\mathbb{P}^8$ (Proposition  \ref{Cubic}), and where $\mc{X}_x$ is the intersection of two quadrics in $\mathbb{P}^5$ (Proposition \ref{AnotherRE}).
 \end{remark}

\subsection{Examples}
\begin{example}\label{GoodBad} Consider $\Bbb{V}=\Bbb{V}(\sL_2,1)$ on $\ovmc{M}_{2}$. By \cite{NR}, for $x\in \mc{M}_{2}$ one has
$\mathbb{V}|^*_x \cong {H}^0(\mc{X}_x,\mc{L}_x)$, where $(\mc{X}_x,\mc{L}_x) \cong (\mathbb{P}^3, \mathcal{O}(1))$. By Corollary \ref{MainLemma} (also see Proposition \ref{Yess}), if $\Bbb{V}$  has geometric interpretation at boundary points
(see Definition \ref{GeometricInterpretation}), then
\begin{equation}\label{sunday}
c_1(\Bbb{V}[m])= {{m+3}\choose {4}} \ c_1({\Bbb{V}})= \frac{(m+3)(m+2)(m+1)m}{24}\cdot c_1(\Bbb{V})
\end{equation}
which we can show fails by intersecting with $\op{F}$-curves.
\begin{remark}
We will however show in Example \ref{GoodBad2} that geometric interpretations do indeed hold along the divisor $\Delta_0$ of $\ovmc{M}_{2}$, and fail along  $\Delta_{1,\emptyset}$.
\end{remark}

There are two types of $\op{F}$-curves on $\ovmc{M}_{2}$: (1) the image of a map from $\ovmc{M}_{0,4}$ for which points are identified in pairs; (2) the image of a map from $\ovmc{M}_{1,1}$ given by attaching a point $(E,p) \in \mc{M}_{1,1}$, gluing the curves at the marked points.    We see a contradiction when we intersect with either type of $\op{F}$-curve.

\subsubsection*{Intersecting both sides of Equation \ref{sunday} with the $\op{F}$-curve defined by $\ovmc{M}_{0,4}$}
By the formulas of Fakhruddin \cite{Fakh}, we can see that if $D_m=0$, then the intersection of Equation \ref{sunday} with a pigtail type $\op{F}$-curve would imply the following identity:
\begin{equation}\label{pigtail}
\sum_{0\leq \lambda,\mu\leq m}\deg\Bbb{V}(\sL_2,\{\lambda,\lambda,\mu,\mu\},m)=\frac{(m+3)(m+2)(m+1)m}{24}.
\end{equation}
Here we note $\lambda^*=\lambda$ for $\sL_2$.
But Equation~\ref{pigtail} doesn't hold, for example, for $m=2$, the right hand side equals $5$, whereas there are three non-zero terms on the left hand side:
\begin{multline}\deg\Bbb{V}(\sL_2,\{\omega_1,\omega_1,2\omega_1,2\omega_1\},2)= \deg\Bbb{V}(\sL_2,\{2\omega_1,2\omega_1,\omega_1,\omega_1\},2)\\=1,\ \deg\Bbb{V}(\sL_2,\{2\omega_1,2\omega_1,2\omega_1,2\omega_1\},2)=2,\end{multline}
which add up to $4$.  We note that using Witten's dictionary, Theorem \ref{WD}, one can verify that
 $$\sum_{0\leq \lambda,\mu\leq m}\rk\Bbb{V}({\sL_2,\{\lambda,\lambda,\mu,\mu\},m})=\frac{(m+3)(m+2)(m+1)}{6}.$$

\subsubsection*{Intersecting both sides of Equation \ref{sunday} with the $\op{F}$-curve defined by $\ovmc{M}_{1,1}$}\label{GoodGood}
If Equation \ref{sunday} held, then by pulling back along the map from $\ovmc{M}_{1,1}$ onto the second type of  $\op{F}$-curve, one would have that
\begin{multline}\label{sundayEE}
\sum_{\mu=0}^m r^{(m)}_{\mu}c^{(m)}_{\mu^*}=  \frac{(m+3)(m+2)(m+1)m}{24}\cdot (r^{(1)}_1c^{(1)}_1 +r^{(0)}_0c^{(0)}_0)
=-\frac{(m+3)(m+2)(m+1)m}{24}\cdot 2\cdot \frac{1}{6},
\end{multline}
where $r^{(m)}_{\mu}$ and $c^{(m)}_{\mu}$ are ranks and first Chern classes of the bundle $\Bbb{V}(\sL_2,\mu,m)$ on $\ovmc{M}_{1,1}$, respectively.
We know that $r^{(m)}_{\mu}=0$ if $\mu$ is odd and equals $m+1-\mu$ if $\mu$ is even. The degree  $c^{(m)}_{\mu}=c^{(m)}_{\mu^*}$ is given by \cite[Corollary 6.2]{Fakh}: if $\mu$ is even,
$c^{(m)}_{\mu}= -\bigl(\frac{\mu^2-3m\mu+2m^2-\mu+2m }{24}\bigr).$
Consider $m=2$, then
$$\sum_{\mu=0}^m r^{(m)}_{\mu}c^{(m)}_{\mu^*}= r_0^{(2)}c_0^{(2)} + r_2^{(2)}c_2^{(2)}= -3 \frac{1}{2}-1 \frac{2}{24}=-\frac{19}{12}.$$
The right hand side  of Equation \eqref{sundayEE} is $-\frac{10}{6}$.  We again see that Equation \eqref{sunday}  fails.

 \end{example}

\begin{example}\label{ForwardReference} The bundle $\Bbb{V}=\Bbb{V}(\sL_2,\{\omega_1^n\},1)$ on $\ovmc{M}_{2,n}$ if $n$ is even
 has the property that for $x\in \mc{M}_{2,n}$ one has
$\mathbb{V}|^*_x \cong {H}^0(\mc{X}_x,\mc{L}_x)$, where $(\mc{X}_x,\mc{L}_x) \cong (\mathbb{P}^3, \mathcal{O}(1))$. We note that $\Bbb{V}(\sL_2,1)$
on  $\ovmc{M}_{2}$ does not pullback to $\Bbb{V}=\Bbb{V}(\sL_2,\{\omega_1^n\},1)$.  If $\Bbb{V}$  has geometric interpretation at boundary points, then
\begin{equation}\label{sundaynew}
c_1(\Bbb{V}[m])= {{m+3}\choose {4}} \ c_1({\Bbb{V}})= \frac{(m+3)(m+2)(m+1)m}{24}\cdot c_1(\Bbb{V})
\end{equation}
which we show fails by intersecting with the $\op{F}$-curve given by the image of a clutching map from $\ovmc{M}_{0,4}$ for which the first two points are identified, and the second two points are attached to a point in $\ovop{M}_{0,n+2}$.  The calculation is
analogous to the one above, the identity failing already at $m=2$.

\end{example}

\begin{remark}\label{Tony}
By \cite[Theorem E]{KP}, given $\mathbb{V}$ on $\ovmc{M}_{g,n}$ with $n=0$, the corresponding polarized varieties $(\mc{X}_x,\mc{L}_x)$ cannot be of Delta invariant zero, for  genus $g \ge 3$, $x\in \mathcal{M}_g$.
\end{remark}

\section{Conjectural formulas in non $\Delta$-invariant $0$ examples}\label{HigherGenusEvidence}

To  illustrate that the techniqes introduced in this work are applicable in a context more general than for the class of bundles with $\Delta$-invariant zero rank scaling, we work in three other situations where there is a conjectural description of the polarized pairs $(\mc{X}_x, \mathcal{L}_x)$ at points $x$ on the boundary of $\ovmc{M}_{g,n}$.

We note that Theorem \ref{MainLemmaI} gives a stronger obstruction than what is obtained here because as we have stated earlier, if $\Delta((\mc{X}_x, \mathcal{L}_x))=0$, for some $x$, and if there are geometric interpretations at points on the boundary, then given  $(\mc{X}_x, \mathcal{L}_x)$ for any $x \in \ovmc{M}_{g,n}$,  the variety $\mc{X}_x$ will be projective of minimal degree, and hence classified. Therefore, if there is failure of the identities predicted, then one may conclude  that geometric interpretations did not actually exist at boundary points.  On the other hand, if the identities from Propositions   \ref{recurso},  \ref{Cubic} or \ref{AnotherRE} fail, it could be either because there are no such polarized pairs $(\mc{X}_x,\mc{L}_x)$ at boundary points $x$, or because if $(\mc{X}_x,\mc{L}_x)$ exists,  assumptions about  the variety $(\mc{X}_x,\mc{L}_x)$  were incorrect.

We also note that for these examples we have used a script CBRestrictor \cite{CBRestrictor}, which compute the ranks of conformal blocks bundles over curves of positive genus, using \cite{ConfBlocks}.

\subsection{Coble's quartic hypersurface}
Given a smooth nonhyperelliptic curve $C$ of genus $3$, one has that
$\mathbb{V}(\sL_2,1)|^*_{[C]} \cong {H}^0(Z_4,\mc{L})$, where $Z_4$ is Coble's quartic hypersurface in $\mathbb{P}^7$ \cites{NR3}.  Assuming that geometric interpretations hold at all points of $\ovmc{M}_3-\Delta_1$, in addition that over $\ovmc{M}_3-(\Delta_1 \cup \mc{H}_3)$,  the corresponding ``moduli space" remains a quartic hypersurface in $\Bbb{P}^7$,  we give in Proposition \ref{recurso}, an identity which must be satisfied by the first Chern classes of multiples $\mathbb{V}[m]= \mathbb{V}(\sL_2, m)$.  Using \cite[Theorem 3]{MOP}, we have checked the predicted relations hold in the first $7$ nontrivial cases.

\begin{proposition}\label{recurso}
Let $\mathbb{V}[m]=\mathbb{V}(\sL_2,m)$ on $\ovmc{M}_3$.  Assuming that  for all points  $[C] \in \ovmc{M}_3-\Delta_1$, corresponding to nonhyperelliptic curves $C$,
one has that
$$\mathbb{V}(\sL_2,1)|_{[C]}^*   \cong {H}^0(Z_4,\mc{L}),$$
where $Z_4$ is a quartic hypersurface in $\mathbb{P}^7$, then
$$c_1(\Bbb{V}[m])+D_m= \left(\binom{7+m}{m-1} -\binom{m+3}{8} - 165\binom{m+3}{7}\right)c_1(\Bbb{V}) + \binom{m+3}{7} c_1(\Bbb{V}[4]) $$
and where $D_m$ is a linear combination of the class of the hyperelliptic locus $\mathcal{H}_3$ and $\delta_1$.
\end{proposition}

\begin{proof} We work over $X=\ovmc{M}_3\setminus (\Delta_1\cup \mathcal{H})$, and show that the class of $D_m$ restricted to $X$ is zero.
Let $\mathcal{I}_m$ (with $m\geq 2$) be the kernel of the surjective map
$$\operatorname{Sym}^m\Bbb{V}^*\to \Bbb{V}[m]^*$$
which can be interpreted for $x\in X$ as the kernel of map
$H^0(\Bbb{P}^7,\mathcal{O}(m))\to H^0(\mathcal{X}_x,\mathcal{L}^m)$. Therefore $(\mathcal{I}_m)_x$ is the space of sections
of the ideal sheaf of $\mathcal{X}_x$. Clearly $\mathcal{I}_m$ is a vector bundle over $X$. Note that
\begin{equation}\label{equato}
c_1(\mathcal{I}_m)= c_1(\Bbb{V}[m]) - \binom{7+m}{m-1} c_1(\Bbb{V})
\end{equation}
We also have isomorphisms $\mathcal{I}_4\tensor \operatorname{Sym}^{m-4}\Bbb{V}^*\to \mathcal{I}_m$ (the ideal of a quartic is generated in degree $4$). Taking Chern classes using Equation \eqref{equato}, we get the desired vanishing: $D_m\mid_X=0$.
\end{proof}

\noindent
A careful check, using \cite[Theorem 3]{MOP}, keeping in mind their result gives the slope of the divisor, rather than its class, one obtains:
\begin{multicols}{2}
\begin{enumerate}
\item $c_1(\mathbb{V}[1])=4 \ \lambda-\delta_{irr}$;
\item $c_1(\mathbb{V}[2])=27 \ \lambda-8 \ \delta_{irr}-3 \ \delta_1$;
\item $c_1(\mathbb{V}[3])=108 \ \lambda-37 \ \delta_{irr}-16 \ \delta_1$;
\item $c_1(\mathbb{V}[4])=329 \ \lambda - 128 \ \delta_{irr} -64 \ \delta_1$;
\item $c_1(\mathbb{V}[5])=840 \ \lambda -366 \  \delta_{irr} - 192 \ \delta_1$;
\item $c_1(\mathbb{V}[6])=1890 \ \lambda -912 \  \delta_{irr} - 502  \ \delta_1$;
\item $c_1(\mathbb{V}[7])=3864 \  \lambda- 2046  \  \delta_{irr} - 1152 \ \delta_1$;
\item $c_1(\mathbb{V}[8])=7326 \ \lambda- 4224 \ \delta_{irr} - 2438  \ \delta_1$;
\item $c_1(\mathbb{V}[9])=13068 \  \lambda - 8151 \ \delta_{irr} -4752  \ \delta_1$;
\end{enumerate}
\end{multicols}
\noindent
So that as we expect from the predicted relations:
\begin{itemize}
\item $9c_1(\mathbb{V}[1]) - c_1(\mathbb{V}[2]) = \mathcal{H}+6\delta_1$;
\item $45c_1(\mathbb{V}[1]) - c_1(\mathbb{V}[3]) = 8 \ (\mathcal{H}+\delta_1)$;
\item $826 \ c_1(\mathbb{V}[1]) + c_1(\mathbb{V}[5])- 8 \ c_1(\mathbb{V}[4]) =168 \  \mathcal{H} +824 \ \delta_1$;
\item $4662 \ c_1(\mathbb{V}[1])+c_1(\mathbb{V}[6]) - 36 \ c_1(\mathbb{V}[4]) =966 \  \mathcal{H} +3384 \ \delta_1$;
\item $16842 \ c_1(\mathbb{V}[1])+c_1(\mathbb{V}[7]) - 120 \ c_1(\mathbb{V}[4]) =3528 \  \mathcal{H} +17112 \ \delta_1$;
\item $48180 \ c_1(\mathbb{V}[1])+c_1(\mathbb{V}[8]) - 330 \ c_1(\mathbb{V}[4]) = 10164 \  \mathcal{H} + 49174 \ \delta_1$; and
\item $118305 \ c_1(\mathbb{V}[1])+c_1(\mathbb{V}[9]) - 792 \ c_1(\mathbb{V}[4]) = 25080 \  \mathcal{H} + 121176 \ \delta_1$.
\end{itemize}
Here $\mathcal{H}=9\lambda-\delta_{irr}-3\delta_1$ is the class in the Picard group of the stack $\ovmc{M}_3$, of the effective divisor $\op{H}$ of hyperelliptic curves in $\ovmc{M}_3$. The presence of non-zero coefficients of $\delta_1$ is evidence of a failure of geometric interpretations on the boundary divisor $\Delta_1$.

\subsection{Coble's cubic hypersurface}\label{CobleCubic}
Coble also considered $\mathcal{Z}_3 \hookrightarrow \mathbb{P}^8=\mathbb{P}(\op{H}^0(J^{g-1}C, \mathcal{O}_{JC}(3\theta))$,  for $C$ a smooth curve of genus $2$, a cubic hypersurface \cite{Coble}.  It is known
that  $\mathbb{V}(\sL_3,1)|^*_{[C]} \cong \op{H}^0(SU_{C}(3), \mathcal{L})$, and $SU_{C}(3)$ is a degree $2$ cover of  $\Bbb{P}^8$ branched over a sextic  dual to $\mathcal{Z}_3$ \cite{Ortega,nguyen}.


\begin{proposition}\label{Cubic}Let $\mathbb{V}=\mathbb{V}_{\{0\}}(\sL_3,1)$ on $\ovmc{M}_{2}$.
If geometric interpretations exist for $\mathbb{V}$ at all points of $\ovmc{M}_{2}$, with every $\mathcal{X}_x$
isomorphic to some degree $2$ cover of $\Bbb{P}^8$ branched over a sextic (and the line bundle pulled back from $\Bbb{P}^8$), then
for all $m\geq 1$, one has that
$$c_1(\Bbb{V}[m])=(\binom{m+8}{9} +\binom{m+5}{9}-55\binom{m+5}{8}) c_1(\Bbb{V})+\binom{m+5}{8} c_1(\Bbb{V}[3])+D_m,$$
where $D_m$ is the anomalous divisor supported on the boundary.  In particular, if Conjecture \ref{ExtendZ} is sharp, $D_m$ to be supported entirely on $\Delta_{1,\phi}$.
\end{proposition}
The proof uses the projection formula, and the following fact: For $\pi:Y\to \Bbb{P}^8$ of degree $2$ branched along a sextic, one has an isomorphism $\pi_*\mathcal{O}=\mathcal{O}\bigoplus (L\tensor\mathcal{O}(-3))$ (for some line $L$).

Using \cite{Fakh} and \cite{MOP} we computed $c_1(\mathbb{V}[1])=9\lambda -2 \ \delta_0, c_1(\mathbb{V}[2])= -11\delta_{irr} +9 \delta_1, c_1(\mathbb{V}[3])=332 \ \lambda -94 \ \delta_0-34  \ \delta_1, c_1(\mathbb{V}[4])= 1152 \ \lambda - 361 \ \delta_0- 153  \ \delta_1$ and $ c_1(\mathbb{V}[5])= 3330 \ \lambda -  964 \ \delta_0- 738  \ \delta_1$.

We note that  the relations given in Proposition \ref{Cubic} hold for the first three (nontrivial) cases: $D_2=9\delta_1$; $D_4=c_1(\mathbb{V}[4])+274 \ c_1(\mathbb{V}) - 9 \ c_1(\mathbb{V}[3])=279 \ \delta_1$; and  $D_5=c_1(\mathbb{V}[5])+1750 \ c_1(\mathbb{V}) - 45 \ c_1(\mathbb{V}[3])=1020 \ \delta_1$. The presence of non-zero coefficients of $\delta_1$ provides evidence towards failure of geometric interpretations on the boundary divisor $\Delta_{1,\phi}$.

\subsection{An intersection of two quadrics in $\mathbb{P}^5$}\label{Quadrics}
Here we consider $\mathbb{V}[m]=\mathbb{V}(\sL_2, \{2m\omega_1\}, 2m)$ on $\ovmc{M}_{2,1}$. By  \cite{DR} (and Hecke transforms) $\mathbb{V}|_{x}^* \cong \op{H}^0(\mc{X}_x,\mc{L}_x)$, where $\mc{X}_x$ is the intersection of two quadrics in $\mathbb{P}^5$, when $x=(C,p)$ is a smooth pointed genus $2$ curve. Here we conjecture that the same geometric interpretation holds for any stable $x=(C,p)$.
\begin{proposition}\label{AnotherRE} Suppose that the geometric interpretation of conformal blocks extends to points on the boundary of $\ovmc{M}_{2,1}$, with $(\mc{X}_x,\mc{L}_x)$ an intersection of quadrics in $\Bbb{P}^5$ for all $x\in \ovmc{M}_{2,1} $. Then,
 \begin{equation}\label{anotherre}
c_1(\mathbb{V}[m])= \frac{(m-4)(m-2)m(m+1)(7m-5)}{12} c_1(\mathbb{V}[1])+({{m+3}\choose{5}}-21 {{m+1}\choose{5}})c_1(\mathbb{V}[2]) + {{m+1}\choose{5}}c_1(\mathbb{V}[4]).
\end{equation}
 \end{proposition}

\begin{proof}
 If $X$ is a intersection of two quadrics in $\Bbb{P}=\Bbb{P}^5=\Bbb{P}(V)$, we have a map
 $\mathcal{W}=Q\tensor \mathcal{O}_{\Bbb{P}}(-2)\to \mathcal{O}_{\Bbb{P}}$, $\rk\mathcal{W}=2$.
 which we can view as a section $s$ of $\mathcal{W}^*$. Now $X=Z(s)$ has (pure) dimension $3$, so $s$ corresponds to a regular sequence, and we have a Koszul complex
 $$0\to\wedge^2 \mathcal{W}\to \mathcal{W}\to \mathcal{O}_{\Bbb{P}}\to \mathcal{O}_X\to 0$$
 and hence
 $$0\to L\tensor \mathcal{O}_{\Bbb{P}}(m-4)\to Q\tensor \mathcal{O}_{\Bbb{P}}(m-2)\to \mathcal{O}_{\Bbb{P}}\to \mathcal{O}_X(m)\to 0$$
 where $L=\det Q$. Now take global sections and use standard cohomology vanishing to get an exact sequence:
 $$0\to L\tensor \operatorname{Sym}^{m-4}V^*\to Q\tensor \operatorname{Sym}^{m-2}V^*\to \operatorname{Sym}^{m}V^*\to V[m]^*\to 0$$
 where $V[m]^*=H^0(X,\mathcal{O}_X(m))$. Imagining a family of such $X$'s over a base $S$ we see that $L$ and $Q$ form vector bundles over $S$. Therefore,
 \begin{equation}\label{koszul}
 c_1(\mathbb{V}[m])= ({{m+5}\choose{6}} -2{{m+3}\choose{6}} + {{m+1}\choose{6}}) c_1(\mathbb{V}[1])+{{m+3}\choose{5}}\alpha  + {{m+1}\choose{5}}\beta
 \end{equation}
 where $\alpha=- c_1(Q)$ and $\beta=c_1(L)$ with $\alpha+\beta=0$.

It is easy to see that $\alpha= c_1(\Bbb{V}[2])-7 c_1(\Bbb{V})$ by putting in $m=2$ in Equation \eqref{koszul}. Put $m=4$ in the same equation and get
$c_1(\mathbb{V}[4])= ({{9}\choose{6}} -2{{7}\choose{6}}) c_1(\mathbb{V}[1])+{{7}\choose{5}}\alpha  + \beta$. And hence we
obtain formulas for $\alpha$ and $\beta$, and hence  Equation \eqref{anotherre}. We also verify $\alpha+\beta=0$, as claimed above, using the  numerical formulas below.
\end{proof}

We have checked Equation \eqref{anotherre} in the first four (nontrivial) cases.
To do so,  we  computed the first Chern classes of the first 7 multiples using the formula for the slope of $c_1(\mathbb{V})$ given in \cite[Theorem 3]{MOP}.  These are:

 \begin{enumerate}
 \item $c_1(\mathbb{V}[1])=\frac{9}{2} \ \lambda +3 \  \psi_1 - \frac{5}{4} \ \delta_{irr} -\frac{3}{2}\delta_1$;
 \item $c_1(\mathbb{V}[2])= 19 \ \lambda + 19 \  \psi_1 - 7 \  \delta_{irr} -8 \  \delta_1$;
 \item $c_1(\mathbb{V}[3])= \frac{99}{2} \ \lambda + 66 \  \psi_1 - \frac{91}{4}  \ \delta_{irr} -\frac{51}{2}  \ \delta_1$;
  \item $c_1(\mathbb{V}[4])= 102 \ \lambda + 170 \  \psi_1 - \frac{281}{5} \  \delta_{irr} -\frac{312}{5} \delta_1$;
 \item $c_1(\mathbb{V}[5])= \frac{365}{2} \ \lambda + 365 \  \psi_1 - \frac{469}{4}  \ \delta_{irr} -\frac{259}{2}  \ \delta_1$;
 \item $c_1(\mathbb{V}[6])=297 \ \lambda + 693 \  \psi_1 - 218 \  \delta_{irr} - 240 \  \delta_1$; and
 \item $c_1(\mathbb{V}[7])=\frac{903}{2} \ \lambda +1204 \  \psi_1 - \frac{1491}{4} \  \delta_{irr} - \frac{819}{2} \  \delta_1$.
 \end{enumerate}

With the help of the relation $\lambda=\frac{1}{10}\delta_{irr}+\frac{1}{5}\delta_{1}$,  proved by Mumford in \cite{Mumford},
we were able to verify the following identities are satisfied:
\begin{itemize}
\item $-16 \ c_1(\mathbb{V}[1])+6 \ c_1(\mathbb{V}[2])-c_1(\mathbb{V}[3])=  0$;
\item $225 \ c_1(\mathbb{V}[1])-70  \ c_1(\mathbb{V}[2])+6 \ c_1(\mathbb{V}[4]) -c_1(\mathbb{V}[5])= 0$;
\item $1036 \ c_1(\mathbb{V}[1])- 315  \ c_1(\mathbb{V}[2])+ 21 \ c_1(\mathbb{V}[4]) -c_1(\mathbb{V}[6])= 0$; and
\item $3080 \ c_1(\mathbb{V}[1])-  924 \ c_1(\mathbb{V}[2])+ 56 \ c_1(\mathbb{V}[4]) -c_1(\mathbb{V}[7])= 0$.
\end{itemize}

\section{Chern class identities for each type of projective variety of minimal degree}
In this section we make the identities from Theorem \ref{MainLemma} explicit for each type of bundle with $\Delta$-invariant zero rank scaling.  For those bundles with projective rank scaling, we do a little more, giving both necessary and sufficient conditions for extensions. Examples for each type are given in Section \ref{ExampleSection}.
\subsection{Projective spaces as moduli}

 \begin{definition}\label{Sym}
There is a natural morphism coming from the algebra structure on the sheaf of conformal blocks
$T_m:\Bbb{V}[m]\to \operatorname{Sym}^m(\Bbb{V}),$
after dualizing.
\end{definition}

\begin{lemma}\label{expectedgrowth}
Set $d=\rk \Bbb{V}-1$, and suppose that $\mathbb{V}|_x^* \cong {H}^0(\mc{X}_x,\mc{L}_x)$  for $x\in \mc{M}_{g,n}$. The following are equivalent:
\begin{enumerate}
\item [(a)] $\operatorname{rk} (\Bbb{V}[m])= {{m+d}\choose {d}}$ for all $m$.
\item [(b)]$(\mc{X}_x,\mc{L}_x)=(\Bbb{P}^d,\mathcal{O}(1))$ for all $x\in \mc{M}_{g,n}$.
\item [(c)]The maps $T_m$ are isomorphisms over $\mc{M}_{g,n}$.
\end{enumerate}
\end{lemma}
\begin{remark}\label{FultonRankOne}
If $d=0$, the condition Lemma \ref{expectedgrowth} (a) is automatically true; it is a consequence of a quantum generalization of Fulton's conjecture in representation theory (see \cite{BGMB} and the references therein).  Statement (c) in this case was proved in  \cite[Corollary 2.2]{BGMB}.
\end{remark}

\begin{proof}
 It is easy to see that (a) implies (b) by Proposition \ref{classy},(1). For (c), note that on fibers over $x\in \mc{M}_{g,n}$, $T_m$, from Definition \ref{Sym}, is dual to the maps
$$\phi_m: \operatorname{Sym}^m({H}^0(\mc{X}_x,\mc{L}_x))\to {H}^0(\mc{X}_x,\mc{L}_x^{\tensor m})$$
which are easily verified to be isomorphisms under the assumption (b). So (b) implies (c).
It follows that (c) implies (a) by counting dimensions.
\end{proof}

  Recall (Def. \ref{TOS}, and Proposition \ref{classy}) that $\Bbb{V}$ is said to have  projective rank scaling if it has $\Delta$-invariant zero rank scaling with degree $1$.

\begin{proposition}\label{PropAnom}If $\Bbb{V}$ has  projective rank scaling, then for $m\geq 1$
\begin{equation}\label{stringnew}
{{m+d}\choose {d+1}}c_1(\Bbb{V})= c_1(\Bbb{V}[m])+D_m,
\end{equation}
where $D_m$ is an effective Cartier divisor supported on $\ovmc{M}_{g,n}\setminus \mc{M}_{g,n}$. Note that $D_1=0$.
\end{proposition}

\begin{proof}
It follows from Lemma \ref{expectedgrowth} that if a bundle $\Bbb{V}$ has  projective rank scaling, the maps $T_m$ are isomorphisms on $\mc{M}_{g,n}$. Taking determinants, we find  a map\  $\det \Bbb{V}[m]\!\rightarrow \det \operatorname{Sym}^m(\Bbb{V})$.  We obtain a global section of
$(\det \Bbb{V}[m])^{-1}\tensor \det \operatorname{Sym}^m(\Bbb{V})$, and can write
$(\det \Bbb{V}[m])^{-1}\tensor \det \operatorname{Sym}^m(\Bbb{V})=\mathcal{O}(D_m)$ for an effective Cartier divisor $D_m$ supported on $\ovmc{M}_{g,n}\setminus \mc{M}_{g,n}$. This gives the assertion.
\end{proof}
\begin{definition}\label{Jan1}
We will call  the divisor $D_m$ in Equation \ref{stringnew} the {\bf{projective space scaling anomaly}}.
\end{definition}

We make the following basic observation:
\begin{proposition}\label{Yess}
Suppose $\Bbb{V}$ has projective rank scaling. The following are equivalent.
\begin{enumerate}
\item[(a)] For each $x\in \ovmc{M}_{g,n}$, there exists a pair $(\mc{X}_x,\mc{L}_x)$ of a projective scheme and an ample line bundle such that  $\mathcal{A}_x \cong \bigoplus_{m\geq 0} {H}^0(\mc{X}_x,\mc{L}_x^{\tensor m})$.
\item[(b)] $D_m=0\in H^2(\ovmc{M}_{g,n},\Bbb{Q})$ for all $m\geq 2$.
\item[(c)] The maps $T_m$ are isomorphisms, so that $\Bbb{V}[m]=\operatorname{Sym}^{m}(\Bbb{V})$ for all $m$.
\end{enumerate}
\end{proposition}

\begin{proof}
If (a) holds then as in Lemma ~\ref{expectedgrowth}, $(\mc{X}_x,\mc{L}_x)=(\Bbb{P}^d,\mathcal{O}(1))$ for all $x\in \ovmc{M}_{g,n}$ and
$T_m$ is dual to the maps
$\operatorname{Sym}^m({H}^0(\mc{X}_x,\mc{L}_x))\to {H}^0(\mc{X}_x,\mc{L}_x^{\tensor m})$,
which are isomorphisms since $(\mc{X}_x,\mc{L}_x)=(\Bbb{P}^d,\mathcal{O}(1))$. Therefore $T_m$ is an isomorphism over
$\ovmc{M}_{g,n}$ and $D_m=0$. Therefore (b) and (c) hold.
If (b) holds then since $D_m$ is an effective Cartier divisor, we conclude that $\mathcal{O}(D_m)=\mathcal{O}$ and hence $T_m$ is an isomorphism.
If (c) holds  we can take $(\mc{X}_x,\mc{L}_x)=(\Bbb{P}(\Bbb{V}_x),\mathcal{O}(1))$, and therefore (a) holds.
\end{proof}
\begin{remark}
Let $x\in \ovmc{M}_{g,n}$. The proof of Proposition \ref{Yess} shows that if $\Bbb{V}$ has projective space scaling, and $x\not\in |D_m|$ for all $m\geq 2$, then there exists a pair $(\mc{X}_x,\mc{L}_x)$ of a projective scheme and an ample line bundle such that  $\mathcal{A}_x \cong \bigoplus_{m\geq 0} {H}^0(\mc{X}_x,\mc{L}_x^{\tensor m})$.
\end{remark}

\begin{example}\label{GoodBad2}
In Example \ref{GoodBad} we can compute the projective scaling anomalies $D_m=\alpha(m)\Delta_0 +\beta(m)\Delta_{1,\emptyset}$. By \cite{Mumford}, the intersection numbers of $\Delta_0$ and $\Delta_{1,\emptyset}$, with the pig-tail F-curve, are $-2$ and $1$ respectively; and with the other curve, the numbers are $1$ and $-1/12$ respectively.  Note that in \cite[Corollary 6.2]{Fakh}, the degree of the Hodge bundle is $1/12$.  In Example \ref{GoodBad}, we get the equations $-2\alpha(2)+\beta(2)=5-4=1$ and $\alpha(2)-\frac{1}{12}\beta(2)= -\frac{10}{6}+\frac{19}{12}=-\frac{1}{12}$. Therefore $\alpha(2)=0$ and $\beta(2)=1$.  Using  \cite[Corollary 6.2]{Fakh} and B. Alexeev's formula \cite[Lemma 3.3]{Swin}, we obtain $\alpha(m)=0$,  and that $\beta(m)$
is equal to  $\frac{s(s+1)(2s^2+2s-1)}{6}$ if $m=2s$ is even, and $\frac{s(s+1)^2(s+2)}{3}$ if $m=2s+1$ is odd. Therefore  geometric interpretations extend across $\Delta_0$ (the corresponding factorization is not quasi-rank one  (Definition \ref{SLROnew1})).
\end{example}

\subsection{Quadric hypersurfaces as moduli}\label{QuadricH}
Recall that  if  $(\mc{X}_x,\mc{L}_x)$ is a polarized variety, such that the degree $\mc{L}^{\op{dim}(\mc{X}_x)}=2$ and $\Delta(\mc{X}_x,\mc{L}_x)=0$, then by \cite[Prop 3.1.2]{BeltraSom},
$(\mc{X}_x,\mc{L}_x)=(Q,\mathcal{O}_Q(1))$ where $Q$ is a  quadric hypersurface in projective space.  In this section we consider rank scaling properties and divisor class identities governing such bundles.

\begin{lemma}\label{QEquiv}Set $d=\rk \Bbb{V}-2$, and suppose that $\mathbb{V}|_x^* \cong {H}^0(\mc{X}_x,\mc{L}_x)$  for $x\in \mc{M}_{g,n}$. The following are equivalent:
\begin{enumerate}
\item [(a)] $\operatorname{rk} (\Bbb{V}[m])=2{{m+d-1}\choose {d}} + {{m+d-1} \choose {d-1}}$.
\item [(b)] $(\mc{X}_x,\mc{L}_x)=(Q,\mathcal{O}_Q(1))$, where $Q$ is a (not necessarily smooth) quadric hypersurface in $\mathbb{P}^{d+1}$.
\end{enumerate}
\end{lemma}

\begin{proof} Assume (a). It follows that the dimension of $\mc{X}_x$ is $d$ and the degree $\mc{L}_x^d=2$. Therefore $\Delta(\mc{X}_x,\mc{L}_x)=d +2-(d+2)=0$, and (b) follows from  \cite[Prop 3.1.2]{BeltraSom}.
The implication (b)$\implies$(a) follows from the exact sequence
\begin{equation}\label{ses}
0\to \mathcal{O}_{\Bbb{P}^{d+1}}(-2)\leto{\cdot Q} \mathcal{O}_{\Bbb{P}^{d+1}}\to\mathcal{O}_{Q}\to 0.
\end{equation}
\end{proof}

Recall (Def. \ref{TOS}, and Proposition \ref{classy})) that $\Bbb{V}$ is said to have  quadric rank scaling if it has $\Delta$-invariant zero rank scaling with degree $2$.

 \begin{corollary}\label{QuadricIdentity} If  $\mathbb{V}$ has quadric rank scaling and if a geometric interpretation exists for $\mathbb{V}$ on the boundary of $\ovmc{M}_{g,n}$, then
   \begin{multline}c_1(\Bbb{V}[m])= \Bigl( {{m+d+1}\choose{d+2}} - {{m-2+d+1}\choose{d+2}} - (d+3){{m-2+d+1}\choose{d+1}} \Bigr)c_1(\Bbb{V})
   + {{m-2+d+1}\choose{d+1}} c_1(\Bbb{V}[2]).\end{multline}
 \end{corollary}

 \begin{proof}
 This follows from Corollary \ref{MainLemma}. We indicate an alternate direct proof here.
 Define $\mathcal{I}_m$ to the kernel of $\mu_m:\op{Sym}^m\Bbb{V}^*\to \Bbb{V}[m]^*$. There are maps $\nu_m:\mc{I}_2\tensor\op{Sym}^{m-2}\Bbb{V}^*\to \mc{I}_{m}$. Note that $\mu_m$ and $\nu_m$ are defined on $\ovmc{M}_{g,n}$.
 In general $\mathcal{I}_m$ may not be vector bundles.

 If geometric interpretations exist, $\nu_m$ are isomorphisms for all $m>2$, and $\mu_m$ are  surjections for all $m$ by working over fibers, and using the exact sequence \eqref{ses} (tensored with $\mathcal{O}_{\Bbb{P}^{d+1}}(m)$). Furthermore $\mathcal{I}_2$ would be a line bundle on $\ovmc{M}_{g,n}$.
 Therefore we obtain formulas for $m>2$
 $$c_1(\mc{I}_2)= c_1(\op{Sym}^{2}\Bbb{V}^*)-c_1(\Bbb{V}[2]^*)$$
 and
 $$c_1(\mc{I}_2)\rk(\op{Sym}^{m-2}\Bbb{V}^*) + c_1(\op{Sym}^{m-2}\Bbb{V}^*) =   c_1 (\mc{I}_m)= c_1(\op{Sym}^m(\Bbb{V})^*)-c_1(\Bbb{V}[m]^*)$$

 Putting all these together gives us the desired formula for $c_1\Bbb{V}[m]$.
\end{proof}

\subsection{Veronese surfaces as moduli}

\begin{lemma}\label{VSEquiv}Suppose that $\mathbb{V}|_x^* \cong {H}^0(\mc{X}_x,\mc{L}_x)$  for $x\in \mc{M}_{g,n}$. The following are equivalent:
\begin{enumerate}
\item [(a)] $\operatorname{rk} (\Bbb{V}[m])={{2m+2}\choose {2}}= (m+1)(2m+1)$;
\item [(b)] Either $\mc{X}_x$ is smooth and  $(\mc{X}_x,\mc{L}_x)=(\Bbb{P}^2,\mathcal{O}(2))$, or $(\mc{X}_x,\mc{L}_x)$ is a generalized cone over a rational normal curve of degree $4$ in $\Bbb{P}^4$.
\end{enumerate}
\end{lemma}
\begin{proof}
It is easy to see that (b) implies (a). Assume (a). It follows that $\mc{X}_x$ is two dimensional, the degree $\mc{L}_x^2=4$, and $\op{H}^0(\mc{X}_x,\mc{L}_x)=6$, so $\Delta(\mc{X}_x,\mc{L}_x)= 0$. Now (b) follows from Proposition \ref{classy}.
\end{proof}

\begin{corollary}\label{P2O2}  If  $\mathbb{V}$ has Veronese surface rank scaling and if a geometric interpretation exists for $\mathbb{V}$ on the boundary of $\ovmc{M}_{g,n}$, then for all $m\geq 1$,
$$c_1(\Bbb{V}[m])= A_1(m) c_1(\Bbb{V}) + A_2(m) c_1(\Bbb{V}[2]) + A_3(m) c_1(\Bbb{V}[3]) + A_4(m) c_1(\Bbb{V}[4]),$$
where
\begin{enumerate}
\item $A_1(m)= -7{{m+3}\choose {5}} +20 {{m+2}\choose {5}}-23{{m+1}\choose {5}}-6{{m+3}\choose {6}} +8{{m+2}\choose{6}}-3{{m+1}\choose {6}} +{{m+5}\choose {6}}$;
\item $A_2(m) ={{m+3}\choose{5}}-6{{m+2}\choose{5}} +15{{m+1}\choose{5}}$;
\item $A_3(m)={{m+2}\choose {5}}-6{{m+1}\choose{5}}$; and
\item $A_4(m)={{m+1}\choose {5}}.$
\end{enumerate}
\end{corollary}

\subsection{Rational normal scrolls as moduli}\label{scrolls}

Suppose $a_1,\dots, a_d$ are strictly positive integers, and let $\mathcal{E}=\bigoplus_{i=1}^d \mathcal{O}(a_i)$, a vector bundle on $\Bbb{P}^1$. Let $X= S(a_1,\dots,a_d)=\Bbb{P}(\mathcal{E})$, a projective bundle over $\pone$. Let $L=\mathcal{O}(1)$ be the natural ample line bundle
on $X$. It is known that $\Delta(X,L)=0$. Let $D=\sum a_i$ and $N=D+d-1$. Set $V={H}^0(X,L)^*$, One can show that
$\dim V= N+1$, and clearly $X\hookrightarrow \Bbb{P}^N=\Bbb{P}(V)$.  The varieties $S(a_1,\dots,a_d)$ are called rational normal scrolls (See Section A2H in \cite{Eisenbud}).

When some $a_i$ are zero and the rest positive, $\mc{O}(1)$ is base point free but not ample on $\Bbb{P}(\mathcal{E})$. The image in $\Bbb{P}({H}^0(\Bbb{P}(\mathcal{E}),\mc{O}(1)))$ is again denoted by $S(a_1,\dots,a_d)$.  In this case,
$\Delta( S(a_1,\dots,a_d),\mathcal{O}(1))=0$.

The vector bundles $\bigoplus_{i=1}^d \mathcal{O}(a_i)$ with fixed $d$ and $\sum a_i$ lie in the same component of the moduli stack of bundles on $\pone$. Therefore, the polarized varieties $(S(a_1,\dots,a_d),\mathcal{O}(1))$ are deformation equivalent
for fixed $d$ and $\sum a_i$, and such varieties will have the same rank sequence.

 One can  check that for $m\geq 0$, one has $\op{H}^0(S(a,b),\mathcal{O}(m)) = (m+1)(1 + \frac{m(a+b)}{2})$.

 \smallskip
 The following is a special case of Corollary \ref{MainLemma}:

\begin{corollary}\label{scrollexample}Suppose $\Bbb{V}$ satisfies $(S(1,2),\mathcal{O}(1))$ scaling so that $\rk \Bbb{V}[m]= (m+1)(1+\frac{3m}{2})$
for all positive integers $m$. Assume geometric interpretations exist at boundary points.  Then
for all $m\geq 1$, we have $c_1(\Bbb{V}[m])= A_1(m) c_1(\Bbb{V}) + A_2(m) c_1(\Bbb{V}[2]) + A_3(m) c_1(\Bbb{V}[3])$, \ \
where
\begin{enumerate}
\item $A_1(m)={{m+4}\choose{5}} -6{{m+2}\choose{4}}+12{{m+1}\choose{4}}- 3{{m+2}\choose{5}} +2{{m+1}\choose{5}}$;
\item $A_2(m) ={{m+2}\choose{4}}-5 {{m+1}\choose{4}}$; and
\item $A_3(m)={{m+1}\choose{4}}$.
\end{enumerate}
\end{corollary}


One says that $\Bbb{V}$ satisfies $(\Bbb{P}^1,\mathcal{O}(d))$ scaling if $\rk \Bbb{V}[m]= dm+1$
for all positive integers $m$.

\begin{corollary}\label{TwistedCubicCor} If $\Bbb{V}$ satisfies $(\Bbb{P}^1,\mathcal{O}(3))$ scaling, and if geometric interpretations exist at boundary points,
\begin{equation}\label{Twisted}
c_1(\Bbb{V}[m])= A_1(m) c_1(\Bbb{V}) + A_2(m) c_1(\Bbb{V}[2]) + A_3(m) c_1(\Bbb{V}[3]),
\end{equation}
where
 $A_1(m)= {{m+3}\choose {4}} - 5 {{m+1}\choose {3}} -3{{m+1}\choose {4}} +8{{m}\choose {3}} + 2{{m}\choose {4}}$,
 $A_2(m) ={{m+1}\choose {3}}-4 {{m}\choose {3}}$, and
 $A_3(m)={{m}\choose {3}}$.

\end{corollary}

\section{Towards extension criteria}
In this section we turn our attention towards finding, in limited cases, conditions which guarantee that geometric interpretations of conformal blocks hold at boundary points.
For this we need Definition~\ref{DRankScaling} together with a few additional terms.
\subsection{Notation for Theorem \ref{precisQ}}
We recall that the weights that are used at attaching points when ``factorizing" a bundle $\mathbb{V}(\mathfrak{g}, \vec{\lambda}, \ell)$
at  a point $(C_0; \vec{p})$ using  Theorem \ref{Factorization}, are called the restriction data  for $\mathbb{V}(\mathfrak{g}, \vec{\lambda}, \ell)$
at  $(C_0; \vec{p})$ (see Definition \ref{RestrictionData}).

\begin{definition}\label{FA1}Given a bundle  $\mathbb{V}=\mathbb{V}(\sL_{r+1}, \vec{\lambda}, \ell)$  on $\ovmc{M}_{g,n}$, we say that the {\bf{restriction data for $\mathbb{V}$ is free}},
 if given any boundary point $x \in \Delta_{g_1,J}=\Delta_{g-g_1,J^c}$ or $x\in \Delta_{\op{irr}}$, and $\alpha_1$, $\ldots$, $\alpha_P$ the $P$ weights of restriction data for  $\mathbb{V}(\sL_{r+1}, \vec{\lambda}, \ell)$ at $x$, then
$$\sum_{i=1}^Pa_i \alpha_i=0 \ \implies \  \sum_{i=1}^Pa_i  \ne 0.$$
\end{definition}

For Definition \ref{SLROnew1} recall Factorization, stated in Theorem \ref{Factorization}.
\begin{definition}\label{SLROnew1}  Suppose that $\mathbb{V}=\mathbb{V}(\sL_{r+1}, \vec{\lambda}, \ell)$  on $\ovmc{M}_{g,n}$ is a bundle of rank $R$ with $\Delta$-invariant zero rank scaling.  We say that $\Bbb V$ satisfies {\bf quasi rank one factorization}  if given any boundary point $x \in \Delta_{g_1,J}=\Delta_{g-g_1,J^c}$, (or $x \in \Delta_{\op{irr}}$), and $\alpha_1$, $\ldots$, $\alpha_P$ the $P$ weights of restriction data for  $\mathbb{V}(\sL_{r+1}, \vec{\lambda}, \ell)$ at $x$, then at most one of the $P$ restriction factors
$$\op{rk}\mathbb{V}(\sL_{r+1}, \lambda(J) \cup  \{\alpha_i\}, \ell)\op{rk}\mathbb{V}(\sL_{r+1}, \lambda(J^c) \cup \{\alpha_i^{*}\},\ell),  \ \mbox{ if} \ x \in \Delta_{g_1,J}=\Delta_{g-g_1,J^c}$$
or
$$\op{rk}\mathbb{V}(\sL_{r+1}, \vec{\lambda}\cup \{\alpha_i, \alpha_i^*\}, \ell),  \ \mbox{ if} \ x \in \Delta_{irr}$$
is greater than one.
\end{definition}

\begin{definition}\label{Socle}Suppose that $\mathbb{V}=\mathbb{V}(\sL_{r+1}, \vec{\lambda}, \ell)$  on $\ovmc{M}_{g,n}$ is a bundle of rank $R$ with $\Delta$-invariant zero rank scaling, and such  that $\Bbb V$ satisfies quasi rank one factorization.   If for some  boundary point $x \in \Delta_{g_1,J}=\Delta_{g-g_1,J^c}$, (or $x \in \Delta_{\op{irr}}$), and if for one of the
$P$ weights of restriction data  $\alpha_1$, $\ldots$, $\alpha_P$  of   $\mathbb{V}(\sL_{r+1}, \vec{\lambda}, \ell)$ at $x$, one has that one of the restriction factors
$$\op{rk}\mathbb{V}(\sL_{r+1}, \lambda(J) \cup  \{\alpha_i\}, \ell)\op{rk}\mathbb{V}(\sL_{r+1}, \lambda(J^c) \cup \{\alpha_i^{*}\},\ell),  \ \mbox{ if} \ x \in \Delta_{g_1,J}=\Delta_{g-g_1,J^c}$$
or
$$\op{rk}\mathbb{V}(\sL_{r+1}, \vec{\lambda}\cup \{\alpha_i, \alpha_i^*\}, \ell),  \ \mbox{ if} \ x \in \Delta_{irr}$$
is greater than one,
then that bundle of rank greater than one is called  the {\bf{socle}} of $\mathbb{V}$.
\end{definition}

For a description of $\Delta$-invariant zero rank scaling, see Definition \ref{DRankScaling}.

\begin{definition}\label{SocleScaling}Suppose that $\mathbb{V}=\mathbb{V}(\sL_{r+1}, \vec{\lambda}, \ell)$  on $\ovmc{M}_{g,n}$ is a bundle of rank $R$ which satisfies $\Delta$-invariant zero rank scaling. We say that {\textbf{each socle satisfies $\Delta$-invariant zero rank scaling with the same degree as that of $\Bbb{V}$}}, if
\begin{enumerate}
\item For any (generic) boundary point $x \in \Delta_{g_1,J}=\Delta_{g-g_1,J^c}$, with socle given by restriction data  $\alpha_1$,
 the polarized variety $\mathcal{X}_x(\sL_{r+1}, \lambda(J) \cup  \{\alpha_1\}, \ell)\ \times\  \mathcal{X}_x(\sL_{r+1}, \lambda(J^c) \cup \{\alpha_1^{*}\},\ell)$ with the corresponding product of ample line bundles $\mathcal{L}_x(\sL_{r+1}, \lambda(J) \cup  \{\alpha_1\}, \ell)\boxtimes \mathcal{L}_x(\sL_{r+1}, \lambda(J^c) \cup \{\alpha_1^{*}\},\ell)$
has $\Delta$-invariant zero with the same degree as that of $\Bbb{V}$.
\item For any (generic)  $x \in \Delta_{\op{irr}}$, with socle given by restriction data  $\alpha_1$, the polarized variety $(\mathcal{X}_x(\sL_{r+1}, \vec{\lambda}\cup \{\alpha_1, \alpha_1^*\}, \ell),\mathcal{L}_x(\sL_{r+1}, \vec{\lambda}\cup \{\alpha_1, \alpha_1^*\}, \ell))$ has $\Delta$-invariant zero with the same degree as that of $\Bbb{V}$.
\end{enumerate}
\end{definition}

\subsection{Extension criteria in the case of projective rank scaling}
Our main result is the following:

\begin{theorem}\label{precisQ}
Given a conformal blocks bundle  $\Bbb{V}$  on $\ovmc{M}_{g,n}$ such that
\begin{enumerate}
\item $\mathbb{V}$ has $\Delta$-invariant 0 rank scaling (Def~\ref{DRankScaling});
\item the restriction data for $\Bbb{V}$ is free (Def~\ref{FA1});
\item $\Bbb{V}$ satisfies quasi rank one factorization (Def~\ref{SLROnew1}), and
\item each socle satisfies $\Delta$-invariant 0 rank scaling with the same degree as
$\Bbb{V}$ (Def \ref{SocleScaling}).
\end{enumerate}
Then $\mathbb{V}$ has geometric interpretations at generic boundary points $x \in \ovmc{M}_{g,n}\setminus \mc{M}_{g,n}$, and therefore the corresponding first Chern class scaling identities from Equation \eqref{MainEq} hold on $\ovmc{M}_{g,n}$.
\end{theorem}

The proof proceeds by showing  that the natural  maps $T_m:\Bbb{V}[m]\to \operatorname{Sym}^m (\Bbb{V})$ are  isomorphisms for all $m\geq 0$ on fibers over all points $x\in \ovmc{M}_{g,n}$.  First,  by Lemma \ref{expectedgrowth}, it suffices to show that $T_m$ is an isomorphism over marked curves which have
exactly one node.   Second, since the two sides of $T_m$ have the same ranks, we only need to show that the map is an injection.
We make the argument for points in $\Delta_{g_1,J}=\Delta_{g-g_1,J^c}$, as the argument for points in $\Delta_0$ is analogous.  For simplicity we write $\Delta_J$ for $\Delta_{g_1,J}$.
If  $x \in \Delta_J$ and the curve corresponding to $x$ is $C_1\cup C_2$ with ``normalization" $\tilde{C}= C_1\sqcup C_2$ with two extra marked points $a$ and $b$ we have by Factorization, Theorem \ref{Factorization},
$$\Bbb{V}[m]|_x =\bigoplus_{\mu\in P_{m\ell}(\sL_{r+1})}\widetilde{\Bbb{V}}(\sL_{r+1},\{m\vec{\lambda}, \mu,\mu^*\},m\ell)|_y$$
where $y$ is the pointed curve $\tilde{C}$.  The $\tilde{C}$ are blocks for $\tilde{C}$; since $\tilde{C}$ is disconnected, the $\tilde{C}$ are tensor products of blocks for $C_1$ and $C_2$.
There are maps
$\widetilde{\Bbb{V}}(\sL_{r+1},\{m\vec{\lambda}, \mu,\mu^*\},m\ell)^*\tensor \widetilde{\Bbb{V}}(\sL_{r+1},\{m'\vec{\lambda}, \nu,\nu^*\},m'\ell)^*
\to \widetilde{\Bbb{V}}(\sL_{r+1},\{(m+m')\vec{\lambda}, (\mu+\nu),(\mu+\nu)^*\},(m+m')\ell)^*$,
inducing an algebra structure on $\bigoplus \Bbb{V}[m]|^*_x$ which we denote by $\widetilde{\mathcal{A}}_x$.  Manon showed that the algebra $\widetilde{\mathcal{A}}_x$ is a degeneration of $\mathcal{A}_x$  \cite[Prop 3.3]{Manon}.  Therefore, it suffices to show that the map $\operatorname{Sym}^m(\Bbb{V}|_x)\to \Bbb{V}[m]|_x$ is an isomorphism in the algebra $\widetilde{\mathcal{A}}_x$.

Assuming that the socle exists for our stratum (otherwise take $B=\Bbb{C}$  below),  and that $\widetilde{\Bbb{V}}(\sL_{r+1},\{\vec{\lambda}, \alpha_1,\alpha_1^*\},\ell)$ is of rank $R_1$, we let $\alpha_2,\dots,\alpha_P$ be the other terms in the factorization. Pick non-zero generators $y_2,\dots,y_P$ of the spaces $\widetilde{\Bbb{V}}(\sL_{r+1},\{\vec{\lambda}, \alpha_i,\alpha_i^*\},\ell)^*$, $i>2$. Let $S$ be the conformal blocks algebra of the socle
$$S=\bigoplus_{m\geq 0}\widetilde{\Bbb{V}}(\sL_{r+1},\{m\vec{\lambda}, m\alpha_1,m\alpha_1^*\},m\ell)^*|_y.$$
 S is a graded ring,  the total section ring of  $(M,L)$ ($\tilde{C}$ is smooth). The socle satisfies projective rank scaling, so $(M,L)=(\Bbb{P}^{R_1-1},\mathcal{O}(1))$, and hence
 $S$ is a polynomial algebra in $R_1$ variables. We form a new graded ring $C=S[Y_2,\dots,Y_{P}]$ with graded pieces
$${C}_m= \bigoplus_{m_1+\dots+m_P=m}S_{m_1} Y_2^{m_2}\dots Y_P^{m_P}.$$  Therefore $C$ is a polynomial algebra and the maps
$\operatorname{Sym}^m({C}_1)\to {C}_m$,
are isomorphisms. Now note that the natural map of algebras $C\to \widetilde{\mathcal{A}}_x$ is an isomorphism.  To see this, write $\widetilde{\mathcal{A}}_x=\bigoplus \widetilde{A}_m$. The natural algebra map $C\to \widetilde{\mathcal{A}}_x$ sends:
 $$S_{m_1} Y_2^{m_2}\dots Y_P^{m_P} \ \mapsto  \widetilde{\Bbb{V}}(\sL_{r+1},\{m\vec{\lambda}, \mu,\mu^*\},m\ell)^*|_y,$$
with $\mu=m_1\alpha_1 + m_2\alpha_2+\dots +m_P\alpha_P$, $m=\sum m_i$.   It follows that different tuples
$(m_1,\dots,m_P)$ map to different direct summands in $\widetilde{A}_m$.  The rank of $\widetilde{A}_m$ is the same as that of $C_m$ since both satisfy projective space  rank scaling. It therefore suffices to note  that the map  $S_{m_1} Y_2^{m_2}\dots Y_P^{m_P}$ to $\widetilde{\Bbb{V}}(\sL_{r+1},\{m\vec{\lambda}, \mu,\mu^*\},m\ell)^*|_y$
is injective (by \cite[Proposition 2.1]{BGMB}), and the surjection part  is valid for any genus $g$, since we may replace the use of invariants by the integral  section rings of  line bundles
over  ind-integral affine  Grassmannians.
\subsection{Proof of Theorem \ref{precisQ} in the general case}
 By Lemma \ref{specialization} below, we need to show that $\widetilde{\mathcal{A}}_x$ is the algebra of sections of a polarized variety of $\Delta$-invariant $0$ when $x$ is the generic point of a boundary divisor. If $S$ is the conformal blocks algebra of the socle we will need to verify that the algebra $C$ has the same graded ranks as $\widetilde{\mathcal{A}}_x$. (The map $C\to \widetilde{\mathcal{A}}_x$ is injective for the same reason as in the case of projective rank scaling.)

The graded algebra $C$ is again the algebra of sections of a polarized variety of $\Delta$-invariant zero (see Prop \ref{classy}, Part (5), and the definition of the normalized generalized cone over a projective variety in \cite[Section 1.1.8]{BeltraSom}). We claim that   rank,  degree and dimension of this polarized variety are the same as for $\mathcal{A}_x$. If the triple is $(R_1,D,d_1)$ for the socle, then the triple for the cone is $(R_1+ p-1, D, d_1+p-1)$. Now $R_1+p-1$ equals the rank of $(\mathcal{A}_x)_1$, and the desired equality holds (using the $\Delta$-invariant zero condition).
\begin{lemma}\label{specialization}
Suppose $\widetilde{\mathcal{A}}_x$ is the algebra of sections of a polarized variety of $\Delta$-invariant zero. Then
${\mathcal{A}}_x$ is also the algebra of sections of a (possibly different) polarized variety of $\Delta$-invariant zero.
\end{lemma}
\begin{proof}
Manon has  introduced \cite[Section 3]{Manon} a filtration of the algebra ${\mathcal{A}}_x$, so that the associated graded algebra equals $\widetilde{\mathcal{A}}_x$. Using a Rees algebra construction \cite[see Eq 45, page 16]{Manon}, he forms a (flat) family of algebras $\mathcal{C}_{\bullet}$ parametrized by $\Bbb{A}^1$ such that for $0\neq t\in \Bbb{A}^1$, $\mathcal{C}_{t} \cong {\mathcal{A}}_x$, and $\mathcal{C}_0 \cong \widetilde{\mathcal{A}}_x$.

Since $\widetilde{\mathcal{A}}_x$ is generated in degree $1$, by Nakayama's lemma and the fact that the algebras are constant for $t\neq 0$, we see that $\mathcal{C}_{\bullet}$ is generated in degree $1$. We can therefore form $\Bbb{P}= \operatorname{Proj}(\mathcal{C}_{\bullet})$ over $\Bbb{A}^1$. Since  $\mathcal{C}_{\bullet}$ is a flat sheaf of algebras over $\Bbb{A}^1$, we get a flat family $\pi:\Bbb{P}\to \Bbb{A}^1$ with a relatively ample line bundle $\mathcal{O}(1)$, and  maps  $\mathcal{C}_m \to \pi_*\mathcal{O}(m)$.

The higher cohomology ${H}^i(\pi^{-1}(t),\mathcal{O}(m)|_{\pi^{-1}(t)}),\ i>0,\ m\geq 0$ vanishes for $t=0$ because of our assumption on $\Delta$ invariants (and Remark \ref{vanish}) the polarized variety carrying the algebra $\widetilde{\mathcal{A}}_x$ is necessarily the fiber of $\pi$ over $t=0$, and hence for all $t$ by semi-continuity. The map $\mathcal{C}_m \to \pi_*\mathcal{O}(m)$ is an isomorphism on fibers
at $t=0$, and hence in a neighborhood of $t=0$, and hence over $\Bbb{A}^1$. Therefore ${\mathcal{A}}_x$ which is integral, is the algebra of sections of the polarized variety $(\Bbb{P}_t,\mathcal{O}(m)|_{\Bbb{P}_t})$ for $t\neq 0$. The $\Delta$-invariant is constant in this family, and the desired statement follows.
 \end{proof}

\begin{remark}  If $\mathbb{V}$   is a bundle of rank $\op{R}$ such that
$\op{rk} \Bbb{V}[m] = \ {{m+R -1}\choose {R-1}}$, then one can use  \cite{Fakh} to reduce the proof of Theorem \ref{precisQ} to the statement for $n=4$ by showing that divisors on  both sides of the purported identity intersect all $\op{F}$-curves in the same degree.
\end{remark}

\begin{example}\label{MoreGoodBad}We saw  in Example \ref{GoodBad}, that for $\Bbb{V}=\Bbb{V}(\sL_2, 1)$ on $\ovmc{M}_{2}$, the divisor scaling identity fails,  and so the answer to Question \ref{TheBigOne} is no for this bundle.  To see that $\mathbb{V}$ has a socle which fails to have projective  rank scaling, we find the restriction data a generic point $x$  in $\Delta_{1, \emptyset}$, which can be represented by nodal curve where each arm has genus one.  There is one piece of valid restriction data at $x$, given by $\alpha=0$, and when restricted to each arm, the bundle $\mathbb{V}(\sL_2, \{0\}, 1)$ has rank 2 \cite{Fakh}. The polarized variety associated to the socle is $(\pone\times\pone,\mathcal{O}(1)\boxtimes\mathcal{O}(1))=(Q,\mathcal{O}(1))$ where $Q$ is a quadric in $\Bbb{P}^3$. This polarized variety is of $\Delta$-invariant zero, and degree $2$, and not $1$ as required for an application of Theorem \ref{precisQ}.
 \end{example}

\section{Examples with $\Delta$-invariant zero}\label{ExampleSection}
Here we give examples of families of bundles of $\Delta$-invariant zero of all types.   For $g=0$,  such examples can be found using the Macaulay 2 \cite{M2}, with packages \cite{Schur}, and \cite{ConfBlocks},  the latter of  which implements the formulas of Fakhruddin \cite{Fakh}.   Often long and combinatorial Schubert calculus arguments were used to show all  multiples $\mathbb{V}[m]$ had ranks predicted by experimentation.   We include these as they are representative.

\begin{remark}\label{plussing}In these and other examples we  use the operation of plussing, described in Def~\ref{plusDef}, which can often turn a quantum computation of ranks into a classical one.   We note that while plussing preserves the rank of a bundle, it does not preserve other invariants such as the
 first Chern class.
\end{remark}
\begin{definition}\label{plusDef}
Let $\sigma_j: \omega_i\mapsto \omega_{i+j (\operatorname{mod} r+1)}$ be a cyclic permutation of the affine Dynkin diagram of $\sL_{r+1}$. Then
$\rk \Bbb V(\sL_{r+1}, \{\lambda_1, \dots, \lambda_n\}, \ell) = \rk \Bbb V(\sL_{r+1}, \{\sigma_{j_1}\lambda_1, \dots, \sigma_{j_n}\lambda_n\}, \ell)$ if $j_1+\dots+j_n$ is divisible by $r+1$ \cite{FuchsSw}.
We refer to this operation as {\bf{plussing}}.
\end{definition}

\subsubsection{Families of bundles with $(\mathbb{P}^d,\mathcal{O}(1))$ scaling for $d=0$}
By a quantum generalization of Fulton's conjecture \cite{fultonconjecture} and \cite[Remark 8.5]{qdeform}, if $\mathbb{V}$ is  any bundle  of rank one, then $\rk(\mathbb{V}[m])$ will also be one, and so $\mathbb{V}$ will have projective scaling for $d=0$.  By \cite[Cor. 2.2]{BGMB}, in this case, one has that $c_1(\mathbb{V}[m])=m c_1 (\mathbb{V})$, so by Proposition \ref{Yess}, will have geometric extensions across the boundary.     These include all level one bundles for $\sL_{r+1}$, as proved in \cite{GiansiracusaGibney}.  In \cite{Kaz} a complete characterization  of $\op{S}_n$-invariant bundles of rank one for $\sL_{n}$ on $\ovmc{M}_{0,n}$ was given.  In \cite{Hobson} all rank one bundles for $\sL_2$ on $\ovmc{M}_{0,n}$ were described, as well as generalizations, those bundles for $\sL_{2m}$ with so-called {\em{rectangular}} weights.

\subsubsection{A family of bundles with  $(\mathbb{P}^d,\mathcal{O}(1))$ scaling for $d=1$}\label{ProjectiveYes}
 If $n>1$ is odd, let $\vec\lambda = \{2k\omega_1, ((2k+1)\omega_1)^{n-1}\}$, and if $n>1$ is even, let $\vec\lambda = \{\omega_1, ((2k+1)\omega_1)^{n-1}\}$. Define
$\mathbb V= \mathbb{V}(\sL_2, \vec\lambda, 2k+1)$.
By Remark~\ref{plussing}, $\rk \mathbb V[m] = \rk  \mathbb{V}(\sL_2, 2km\omega_1, \, m(2k+1))=m+1$.  We note that the restriction data for $\mathbb V$ at a point $x$ in the boundary $\Delta_{irr}$ is given by $\mu_1=\mu_1^*= k\omega_1$ and $\mu_2=\mu_2^*= (k+1)\omega_1$, so it is free,  and we have $\rk \mathbb V (\sL_2, \mu_i, \mu_i^*, \vec\lambda, 2k+1)=1$ for $i=1,2$.  At a point $x \in \Delta_{1,J} =\Delta_{0,J^c}$, one also gets free restriction data (which differs depending on the parity of $J$), and all restrictions have rank one.  Therefore, by Theorem~\ref{precisQ}, $\mathbb V$ satisfies projective space Chern class scaling, i.e. $c_1(m\mathbb V)= {m+1\choose 2} c_1(\mathbb V)$, and there is a geometric interpretation  of the conformal block $\mathbb{V}(\sL_2, \vec\lambda, 2k+1)|_{x}$ for all points $x \in \ovmc{M}_{1,n}$.

\subsubsection{A family on $\ovmc{M}_{0,n}$, for $n \ge 4$ with  $(\mathbb{P}^d,\mathcal{O}(1))$ scaling for arbitrary $d$}\label{ProjectiveYesTwo}On $\ovmc{M}_{0,4}$ let $\mathbb{V}=\mathbb{V}(\sL_{r+1},\{(\omega_{i}+\omega_{r+1-i})^4\},2)$. One can check, using Littlewood-Richardson, that  $\rk \mathbb V[m]={{d+m}\choose m}$, where
 $d=2i \le r+1-2i$ (so $i \le \lfloor \frac{r+1}{4} \rfloor$), so these vector bundles satisfy projective rank scaling. These bundles are $\op{S}_4$-invariant and so there is just one boundary restriction, up to symmetry.  The restriction data at the boundary point is free, given by the $d+1$ weights  $\alpha_j=\omega_j+\omega_{r+1-j}$, where $1 \le j \le d$, and $\alpha_{d+1}=0$. $\mathbb{V}$ has quasi rank one factorization, since  for all $j \in \{1,\ldots, d+1\}$,  $\op{rk}\mathbb{V}(\sL_{r+1}, \{\omega_{i}+\omega_{r+1-i},\omega_{i}+\omega_{r+1-i},\alpha_j\},2)=\op{rk}\mathbb{V}(\sL_{r+1}, \{\omega_{i}+\omega_{r+1-i},\omega_{i}+\omega_{r+1-i},\alpha_j^*\},2)=1$.  We have checked that we may apply Theorem \ref{precisQ} to conclude that first Chern class scaling identities hold for this bundle $\mathbb{V}$, and using Fakhruddin's formula \cite{Fakh}, one can see (by a straightforward, but somewhat involved combinatorial argument), that $\op{deg}(\mathbb{V})=\frac{d(d+1)}{2}={{d+1}\choose2}$.  So
$\op{deg}\Bbb{V}[m]={{m+d}\choose{d+1}}{{d+1}\choose2}.$
For $n=5$, we take $\mathbb{V}(\sL_{r+1}, \{(\omega_{i}+\omega_{r+1-i})^3, \omega_{i-1}+\omega_{r-i}, 2\omega_1\}, 2)$.
For $n \ge 6$, we let  $\lambda_j = \omega_i+ \omega_{r+1-i}$, where $i \le \lfloor \frac{r+1}{4} \rfloor$ for $j=1,\dots,4$. Let $\lambda_i=2\omega_1$ for $j=5, \dots, n-1$, $s\equiv (r-n+6) \op{mod}\  (r+1)$, and  $\lambda_n=2\omega_s$. We can define  $\mathbb{V}=\mathbb{V}(\sL_{r+1},\{\vec\lambda\},2)$.
Since the bundles are obtained by plussing (see Remark \ref{plussing}), we  conclude from the $n=4$ calculation that  in both cases,  $\rk \mathbb{V}[m]={{d+m}\choose m}$, where $d=2i \le r+1-2i$.  Checking restriction data, it is straightforward to verify that one can apply Theorem \ref{precisQ}.
We outline the approach for  $n\geq6$:
\begin{enumerate}
\item $\{1,2,3,4\}\subset I$. In this case the restriction data is given by $\mu=2\omega_k$, where $k \equiv( r+1 - \frac{1}{2}\sum_{j\in I} |\lambda_j| )\op{mod}\  (r+1)$. There is one piece of restriction data, so it is free and just as before we have $\rk \mathbb V(\sL_{r+1}, \{\vec\lambda(I), \mu\}, 2) =d+1$, and $\rk \mathbb V(\sL_{r+1}, \{\vec\lambda(I^C), \mu^*\}, 2) =1$.
An argument shows that $\rk \mathbb V(\sL_{r+1}, \{m\vec\lambda(I), m\mu\}, 2m) ={{d+m}\choose m}$, so that the socle satisfies projective rank scaling.
\item $a \in I$, $b, c, d\not\in I$, where $\{a,b,c,d\}=\{1,2,3,4\}$. In this case the restriction data is given by $\mu=\omega_{i+s}+\omega_{r+1-i+s}$, where the indices are modulo $(r+1)$, and $s$ is such that $\frac{1}{2}\sum_{i\in I, i \neq 1 }|\lambda_i|\equiv (r+1-s)\op{mod}\ (r+1)$.
The restriction data is free, moreover, we have  $\rk \mathbb V(\sL_{r+1}, \{m\vec\lambda(I), m\mu\}, 2m) =1$, and one can conclude  $\rk \mathbb V(\sL_{r+1}, \{m\vec\lambda(I^C), m\mu^*\}, 2m) ={{d+m}\choose m}$.
\item  $\{a,b\} \subset I$, $c, d\not\in I$, where $\{a,b,c,d\}=\{1,2,3,4\}$. In this case the restriction data is given by $\mu_j=\omega_j+s+\omega_{r+1-j+s}$, where indices are taken modulo $(r+1)$,  $0 \le j \le d$, and $s$ is such that $\frac{1}{2}\sum_{i\in I, i \neq 1 }|\lambda_i|\equiv (r+1-s)\op{mod}\ (r+1)$. Restriction data is free, and  for all $j \in \{1,\ldots, d+1\}$ we have
$\op{rk}\mathbb{V}(\sL_{r+1}, \vec\lambda(I),\alpha_j\},2)=\op{rk}\mathbb{V}(\sL_{r+1},\vec\lambda(I^C),\alpha_j^*\},2)=1.$
\noindent
\end{enumerate}

\subsubsection{Quadric scaling on $\ovop{M}_{0,n}$ for  $n\ge 5$.}\label{QuadricExamples}\label{QHS}
Let $\mathbb{V}=\mathbb{V}(\sL_4,\{\omega_1+3\omega_2+\omega_3, 3\omega_1+\omega_2+\omega_3, 2\omega_1+\omega_2+2\omega_3,(7\omega_1)^{n-4},7\omega_s\},7)$, where $s \equiv (-n) \mbox{ mod } 4$.
One computes ranks of $\Bbb{V}[m]$ by noticing that $\mathbb{V}$ is obtained by applying the plussing operation to
$\mathbb{V}=\mathbb{V}(\sL_4,\{\omega_1+3\omega_2+\omega_3, 3\omega_1+\omega_2+\omega_3, 2\omega_1+\omega_2+2\omega_3\},7)$ on $\ovop{M}_{0,3}$ which satisfies
$\operatorname{rk} (\Bbb{V}[m]_3)=2{{m+2}\choose {3}} + {{m+2} \choose {2}}$.   To verify $\mathbb{V}$ has quadric hypersurface rank scaling for $d=3$, and that all the conditions of Claim \ref{precisQ} were met, we computed ranks using  Witten's Dictionary \ref{WD} together with the plussing operation and Littlewood-Richardson counting arguments.
So by Lemma \ref{QEquiv}, one has that $\mathbb{V}^* \cong {H}^0(\mc{X},\mc{L})$, where
$\mc{L}$ embeds $\mc{X}$ as  a quadric hypersurface in $\mathbb{P}^4$.   Given any $x \in \delta_{I}$ for any $I$ there will always be a unique socle with quadric rank scaling.  For example, it will have restriction weight $\mu=0$ at $x\in \delta_{123}$, and
the rank of $\mathbb{V}(\sL_{4},\{\lambda_1,\lambda_2,\lambda_3,0\}, 7)$ has quadric hypersurface scaling (for $d=3$).     Therefore $\mathbb{V}$ satisfies the hypothesis of Theorem \ref{precisQ}, so multiples $c_1(\Bbb{V}[m])$ are governed by first Chern class scaling identities given in Corollary \ref{MainLemma}.

\subsubsection{Rational normal scroll scaling on on $\ovmc{M}_{0,n}$ for $n\ge 5$}We set
$\mathbb{V}=\mathbb{V}(\sL_2, \{(2\omega_1)^4,4\omega_1, (5\omega_1)^{n-6}, 5\omega_s\}, 5)$ with $s \equiv n  \ ( \op{mod}\  (2))$.   One can check that $\mathbb{V}$ satisfies $(S(1,2),\mathcal{O}(1))$ scaling so that $\rk \Bbb{V}[m]= (m+1)(1+\frac{3m}{2})$.  To do so we use Witten's Dictionary, Theorem \ref{WD} and Quantum Pieri on $\ovmc{M}_{0,5}$ and then apply the plussing procedure.  One can use \cite{ConfBlocks} to compute coefficients of the   first Chern classes of multiples $\Bbb{V}[m]$ in the nonadjacent bases on $\ovmc{M}_{0,5}$ given by  $\{\delta_{13},\delta_{14}, \delta_{24}, \delta_{25}, \delta_{35}\}$:
$c_1(\mathbb{V})=2(\delta_{14}+ \delta_{25}+ \delta_{35})$,
 $c_1(\mathbb{V}[2])=9(\delta_{14}+ \delta_{25}+ \delta_{35})$,
$c_1(\mathbb{V}[3])=24(\delta_{14}+ \delta_{25}+ \delta_{35})$, and
$c_1(\mathbb{V}[4])=50(\delta_{14}+ \delta_{25}+ \delta_{35})$.
We calculate
$c_1(\Bbb{V}[4])= 10 c_1(\Bbb{V}) - 10 c_1(\Bbb{V}[2]) +5 c_1(\Bbb{V}[3])$, for example.
so the predicted divisor identity given in Section \ref{scrolls}, holds for $c_1(V[4])$.

\subsubsection{Veronese surface scaling on $\ovmc{M}_{0,n}$}
Let $\mathbb{V}=\mathbb{V}(\sL_2, \{\omega_1, 3\omega_1, 4\omega_1, 4\omega_1, 6\omega_1, (8\omega_{1})^{n-6}, 8\omega_s\}, 8)$, for which $s \equiv n \ ( \op{mod}\  (2))$.
$\mathbb{V}[m]$ can be shown to have rank $(m+1)(2m+1)$.  To do so we work first
on $\ovmc{M}_{0,5}$, with the bundle $\mathbb{V}=\mathbb{V}(\sL_2, \{\omega_1, 3\omega_1, 4\omega_1, 4\omega_1, 6\omega_1\}, 8)$ and argue using Witten's dictionary, Theorem \ref{WD}.  We then use the standard plussing procedure.
In the non-adjacent basis for the Picard group of $\ovmc{M}_{0,5}$ consisting of $\delta_{13}$, $\delta_{14}$, $\delta_{24}$, $\delta_{25}$, and $\delta_{35}$, one can use \cite{ConfBlocks} to check that the first five multiples of $\mathbb{V}$ have the following expressions:
$c_1(\mathbb{V})=\delta_{14}+2\delta_{24}+\delta_{25} + 3\delta_{35}; \ c_1(\mathbb{V}[2])=4 \delta_{14}+11 \delta_{24}+ 4 \delta_{25} + 15 \delta_{35}; \ c_1( \mathbb{V}[3])=10 \delta_{14}+32\delta_{24}+ 10 \delta_{25} + 42 \delta_{35}; \ c_1( \mathbb{V}[4])=20 \delta_{14}+70 \delta_{24}+ 20 \delta_{25} + 90 \delta_{35}$; and
 $c_1( \mathbb{V}[5])=35 \delta_{14}+130 \delta_{24}+ 35 \delta_{25} + 165 \delta_{35}$,
A calculation shows that
$c_1(\mathbb{V}[5])= -15 c_1(\mathbb{V}) + 20 c_1(\mathbb{V}[2]) -15 c_1(\mathbb{V}[3]) + 6 c_1(\mathbb{V}[4])$, as is predicted.
While the identity holds, $\mathbb{V}$ does not satisfy the hypothesis of Theorem \ref{precisQ}.  At a point of $\delta_{123}$, one finds three pieces of restriction data: $\mu \in \{2\omega_1, 4\omega_1, 6\omega_1\}$, and  by \cite{ConfBlocks},
$\op{rk}(\mathbb{V}(\sL_2,\{\lambda_1,\lambda_2,\lambda_3,\mu\},8)=2$.

\subsubsection{Twisted cubic  scaling on $\ovmc{M}_{1,n}$, for $n \ge 1$}

 For $d$ even, let $\mathbb{V}^d_k=\mathbb{V}(\sL_{2}, \{2k\omega_1, ((d+2k)\omega_1)^{n-2}, (d+2k)\omega_s\} , d+2k)$, with $s\equiv n \ ( \op{mod}\  (2))$.    First let $n=1$.   By \cite[p. 27]{Fakh}, if $i$ is even,  the rank of $\mathbb{V}(\sL_{2}, i\omega_1, \ell)$, is
$\ell+1-i$, and one has that $\op{rk}( \mathbb{V}^d_k [m])=dm+1$.   Now by plussing, one has that $\op{rk}( \mathbb{V}^d_k [m])=dm+1$ for all $n \ge 1$.  Returning to the case $n=1$, while ranks of multiples do not depend on $k$, by \cite[Corollary 6.2]{Fakh} degrees do:
$\op{deg}(\mathbb{V}^d_k)=-\frac{m(dm+1)(k+d)}{12}$.
This bundle satisfies the identity for $d=3$ in Equation \ref{Twisted}:
$\frac{48(k-3)-4(k-3)}{12}=\op{deg}(\Bbb{V}^3_k[4])= 4 \op{deg}(\Bbb{V}^3_k) -6 \op{deg}(\Bbb{V}^3_k[2]) +4 \op{deg}(\Bbb{V}^3_k[3]).$
There is one necessary boundary restriction, and one has that $\mathbb{V}^d_k$ has quasi rank one factorization.

\bigskip
\noindent
P.B.: Department of Mathematics, University of North Carolina, Chapel Hill, NC 27599,\\
{{email: belkale@email.unc.edu}}

\vspace{0.1 cm}

\noindent
A.G.: Department of Mathematics, University of Georgia, Athens, GA 30602,\\
{{email: agibney@math.uga.edu}}

\vspace{0.1 cm}

\noindent
A.K.: Department of Mathematics, University of Georgia, Athens, GA 30602,\\
{{email: kazanova@math.uga.edu}}

\end{document}